\documentclass[12pt,leqno]{amsart}

\usepackage[utf8]{inputenc}   
\usepackage[T1]{fontenc}      
\usepackage{geometry}         
\usepackage[all]{xy}
\usepackage{amssymb}
\usepackage{enumerate}
\usepackage{enumitem}
\usepackage{stmaryrd}

\voffset-1cm \markleft{\rm AcAa-algebras} 
\newtheorem{theorem}{Theorem}
\newtheorem{definition}[theorem]{Definition}
\newtheorem{corollary}[theorem]{Corollary}

\newtheorem{lemma}[theorem]{Lemma}

\newtheorem{proposition}[theorem]{Proposition}

\newtheorem{defi}{D\'{e}finition}

\newcommand{\K}{\mathbb K}
\newcommand{\C}{\mathbb C}

\newcommand{\g}{\frak{g}}
\newcommand{\h}{\frak{h}}

\newcommand{\f}{\frak{f}}
\newcommand{\we}{\wedge}
\newcommand{\N}{\mathbb N}
\newcommand{\Z}{\mathbb Z}
\newcommand{\R}{\mathbb R}

\newcommand{\om}{\omega}

\newcommand{\ra}{\rightarrow}

\newcommand{\pf}{\noindent{\it Proof. }}
\newcommand{\ds}{\displaystyle}

\newcommand{\im}{{\rm Im }}

\newcommand{\ad}{\text{\rm ad}}
\newcommand{\Ker}{\text{\rm Ker }}

\usepackage{graphicx}

\title{Contact and 2-compatible Lie algebras}
\author{Elisabeth Remm}
\date{}
\address{Universit\'e de Haute-Alsace, IRIMAS UR 7499, F-68100 Mulhouse, France.}
\email{elisabeth.remm@uha.fr}

\begin{document}

\maketitle

\begin{abstract} A $n$-dimensional Lie algebra $\g(V,\mu)$ is called $2$-compatible if it is isomorphic to a quadratic deformation of a Lie algebra $\g_0=(V,\mu_0)$.  By quadratic deformation we means a formal deformation $\mu_t=\mu_0+t\varphi_1+t^2\varphi_2$ where $\mu_t$ is a Lie algebra on $V \otimes \K[[t]]$. It is equivalent to say that we have the following system $\sum_{i+j \leq 4} \varphi_i \circ \varphi_j= 0$. This notion naturally appears in the theory of classification of contact Lie algebras because any $(2p+1)$-dimensional contact Lie algebra is isomorphic to a quadratic deformation of the Heisenberg algebra $\mathcal{H}_{2p+1}$.
\end{abstract}

\noindent{\bf MSC} : 17A30- 17B05 - 17B08- 53D10

\noindent{\bf Keywords}: Contact Lie algebras. Quadratic deformations. Compatible Lie algebras. 


\bigskip

\noindent{\bf Introduction}

A compatible Lie algebra is a pair of Lie algebras such that any linear combination of these two Lie brackets is still a Lie bracket.    This notion appears in the study of linear Poisson structures and  also in the study of bihamiltonnian structures or of the classical Yang-Baxter equation. See for example \cite{Bay} for an introduction to this notion as well as for applications to mathematical physics.  In this last work,  a cohomological  theory for compatible Lie algebras is studied. This notion
 of compatible Lie algebra is  in relation to infinitesimal deformations, also called linear deformations, of a finite-dimensional Lie algebra.  For example, let us consider the variety $\mathcal{N}ilp_n$ of $n$-dimensional nilpotent complex Lie algebras  and its open set $\mathcal{F}il_n$  which elements are the filiform Lie algebras (that is the nilpotent Lie algebra of a maximal nilindex that is equal to $n-2$). Then from \cite{G.H}, any deformation in $\mathcal{F}il_n$ is isomorphic to a linear deformation and thus to a compatible nilpotent Lie algebra.  This class of compatible nilpotent Lie algebras has been recently studied in  \cite{La} which is  devoted to the classification problem.
 
 In this work, we focus on a more general notion of compatibility, which we will call $2$-compatibility and which is provided by two Lie algebras brackets connected by an alternating bilinear mapping. In fact, this notion corresponds to a quadratic deformation of one of these two brackets. After sketching a general study of this structure, we turn our attention to the case of Lie algebras equipped with a contact form. In fact this work was motivated by the study of this class of Lie algebras  (and the reading of interesting results on compatible Lie algebras). Indeed, it has been shown in \cite{G.R.contact} that every Lie algebra equipped with a contact form, which we will briefly call contact Lie algebra, is a linear or quadratic deformation of the Heisenberg algebra.

 This paper is organized as follows. In Section 1 we define the notion of $2$-compatible Lie algebras and the link with the quadratic deformations. Section 2 concerns the study of the nilpotent case. Finally, in  Section 3, we study contact Lie algebras and  begin the description of this class of Lie algebras for the dimensions $3$ and $5$. This classification reveals a Lie group attached to the existence of a form of contact. In the last section, we describe this group and the Lie algebra associated with it.
 
 In this paper, all vector spaces are over algebraically closed field $\K$ of characteristic $0$ and finite dimensional.

\section{Definition}
\subsection{Definition of a $2$-compatible Lie algebras}
Recall that  an algebra $A=(V,\mu)$ over $\K$  is a $\K$-vector space $V$ equipped with a bilinear product $\mu$. If $\varphi$ and $\psi$ are two bilinear products on a vector space $V$, the composition product is the trilinear application given by
$$\varphi \circ \psi (a,b,c)=\varphi(\psi (a,b),c)+\varphi(\psi(b,c),a)+\varphi (\psi(c,a),b)$$for  $a,b,c \in V$. In particular $\g=(V, \mu)$ is a Lie algebra if the skew-symmetric bilinear application $\mu$ satisfies the Jacobi condition
$$\mu \circ \mu =0.$$
\subsection{Definition}
\begin{definition} Let $V$ be a finite dimensional $\K$-vector space.
Let $\g_1=(V,\mu_0)$ and $\g_2=(V,\varphi_2)$ be two Lie algebras  with the same underlying vector space $V$. The triple $\g=(V,\mu,\varphi_2)$ is called a $2$-compatible Lie algebra if there exists a skew-symmetric bilinear map $\varphi_1: V \otimes V \ra V$ satisfying
\begin{align*}
\delta_{\mu_0} \varphi_1=0, \\
\varphi_1 \circ \varphi_1+\delta_{\mu_0} \varphi_2=0,\\
\delta_{ \varphi_2} \varphi_1=0 
\end{align*}
where $\delta_{\mu_0}$ (resp. $\delta_{ \varphi_2}$ ) is  the cohomological operator associated with the Chevalley-Eilenberg cohomology of $\g_1$ (resp. $\g_2$) with values in $V$ and $\circ$ the composition product. With this operator
$$\delta_{\mu_0} \varphi_1=\mu_0 \circ \varphi_1 + \varphi_1 \circ \mu_0.$$
\end{definition}
To highlight the role of $\varphi_1$, we will also write the $2$-compatible Lie algebra as a quadruplet $\g=(V,\mu_0,\varphi_1,\varphi_2)$.

\noindent{\bf Examples} 
\begin{enumerate}
  \item A compatible Lie algebra, notion defined in \cite{Bay,La}, is a triple $\g=(V,\mu_0,\varphi_1)$ where $\mu_0$ and $\varphi_1$ are Lie brackets on $V$ with the compatible relation $\delta_{\mu_0} \varphi_1=\delta_{\varphi_1} \mu_0 =0.$ This structure can be view as a particular case of $2$-compatible Lie algebra corresponding to $\varphi_2=0$ and $\varphi_1 \neq 0.$
  \item Let $\g=(V,\mu_0,\varphi_1,\varphi_2)$ be a $2$-compatible Lie algebra. If $\varphi_1=0$, then $\g=(V,\mu_0,\varphi_2)$ is also a compatible Lie algebra.
   \end{enumerate}

\subsection{Quadratic deformations}
Let $\h$ be a Lie algebra on $V$ whose Lie bracket is denoted by $\mu_0$. A formal deformation of $\mu_0$  is given by a family
$$\varphi_i: V \otimes V \rightarrow V, \ \ i \in \N$$
satisfying $\varphi_0=\mu_0$ and
$$(D_k): \ \ \ \ \ \sum_{i+j=k, i,j\geq 0} \varphi_i\circ \varphi_j
=0
$$
for each $k \geq 1$.
We associate with this family of bilinear maps the power series  $\mu(X,Y)$ given by
$$\mu(X,Y) = \mu_0(X,Y) + t\varphi_1(X,Y) + t^2\varphi_2(X,Y)+ \cdots$$
Then  $\mu$ is Lie bracket if and only if $(D_k)$ are satisfied for each $k \geq 1$. In particular we have
$$\begin{array}{l}
 \mu_0 \circ \varphi_1+\varphi_1 \circ \mu_0=\delta_{\mu_0}\varphi_1=0,\\
  \varphi_1\circ \varphi_1+\delta_{\mu_0}\varphi_2=0,\\
 \varphi_1\circ \varphi_2+\varphi_2\circ \varphi_1+\delta_{\mu_0} \varphi_3=0,\\
 \varphi_2\circ \varphi_2+\varphi_1\circ \varphi_3+\varphi_3\circ \varphi_1+\delta_{\mu_0} \varphi_4=0,\\
     \end{array}
     $$
The formal series $\mu(X,Y)$ is called a deformation of the Lie bracket $\mu_0$ on $V$. Such deformation is linear if $\varphi_i=0$ for $i \geq 2$ and it is called quadratic if $\varphi_i=0$ for $i \geq 3$. In this last case, the previous conditions  corresponds to the axioms of a $2$-compatible Lie algebra, that is  $(V,\mu_0,\varphi_2)$ is a $2$-compatible Lie algebra. Thus we have a ono-one correspondence between the quadratic deformations of $\mu_0$ and the $2$-compatible Lie algebra based on $\mu_0$.

Remark that a linear deformation corresponds to a structure of compatible Lie algebra: to the deformation $\mu_0+t\varphi_1$ corresponds the compatible Lie algebra $(V,\mu_0,\varphi_1).$ In this case $\g_0=(V,\mu_0)$ and $\g_1=(V,\varphi_1)$ are two Lie algebras. But if we consider two Lie algebra structures on $V$, in general the triple $(V,\mu_0,\varphi_1)$ is not compatible because the axiom $(D_1)$ is not satisfied. For example, in  (\cite{La} Example 2.13), one gives the following example: one considers on $V=\K\{e_1,e_2,e_3\}$ the Lie brackets $\mu_0$ and $\varphi$ given by
$$
\left\{
\begin{array}{l}
    \mu_0(e_1,e_2)=e_1    \\
\varphi(e_1,e_2)=e_3,\  \ \varphi(e_1,e_3)=-e_1, \ \ \varphi(e_2,e_3)=e_2,
\end{array}
\right.
$$
(recall that the non defined bracket are considered to be zero). But $\g=(V,\mu_0,\varphi)$ is not a compatible Lie algebra because
$$\delta_{\mu_0}\varphi(e_1,e_2,e_3)=-  e_1.$$ 
It is also not a $2$-compatible Lie algebra because there is no $\varphi_1$ satisfying the relations $D_i$'s, that is there exits no $\varphi_1$ which is a cocyle for $\mu_0$ and $\varphi$.

\noindent{\bf Remark.}  If $\g=(V,\mu_0,\varphi_1,\varphi_2)$ is a $2$-compatible Lie algbebra, then $\widetilde{\g}=(V,\varphi_2,\varphi_1,\mu_0)$ is also a $2$-compatible Lie algebra. 

\subsection{Subalgebras and ideals of $2$-compatible Lie algebras}

\begin{definition} A subalgebra of a $2$-compatible Lie algebra $\g=(V,\mu,\varphi_1,\varphi_2)$ is a quadruplet $\h=(V_1,\mu,\varphi_1,\varphi_2)$ where $V_1$ is a subspace of $V$  which is stable for the bilinear maps $\mu,\varphi_1,\varphi_2$. 

An ideal $I$ of a $2$-compatible Lie algebra $\g=(V,\mu,\varphi_1,\varphi_2)$ is a subalgebra such that
$$\mu(V,I) \subset I,  \ \varphi_1(V,I) \subset I, \  \varphi(V,I) \subset I.$$
\end{definition}
In particular, the center $\mathcal{Z}(\g)$ is the ideal of $\g$ whose elements $X\in V$ satisfy $\mu(X,Y)=\varphi_1(X,Y)=\varphi_2(X,Y)=0$ for any $Y \in V$. 

An interesting class of $2$-compatible Lie algebras is that of nilpotent algebras. A study of nilpotent compatible algebras is made in \cite{La}. For example, any filiform nilpotent Lie algebra is isomorphic to compatible nilpotent Lie algebra $(\g,\mu,\varphi)$ where $\mu$ is given in an adapted basis $\{X_0,X_1,\cdots,X_n\}$ by
$$\mu(X_0,X_I)=X_{i+1}, \ i=1,\cdots,n-1$$
and $\varphi$ is given by 
$$\varphi(X_i,X_j)=\sum{k \geq i+j}C_{ij}^k X_k.$$
 Let  $\g=(V,\mu,\varphi_1,\varphi)$ be a $2$-compatible Lie algebra. Then $\g$ is nilpotent if the lower central series of  $\mu_t=\mu+t\varphi_1+t^2\varphi_2$ satisfies $\mathcal{C}^N(\mu_t)=0$ for some $N >0$. This lower series is also given by
$$\mathcal{C}^0(\g)=V, \ \mathcal{C}^1(\g)=\mathcal{C}^1(\mu)+ \mathcal{C}^1(\varphi_1)+\mathcal{C}^1(\varphi_2)$$
that is $\mathcal{C}^1(\g)$ is the linear space $\llbracket V,V\rrbracket$ generated by the vectors $\mu(X,Y),\varphi_1(X,Y),\varphi_2(X,Y).$ Thus we define $\mathcal{C}^2(\g)=\llbracket V,\llbracket V,V\rrbracket \rrbracket$ and more generally
$$\mathcal{C}^{p+1}(\g)=\llbracket V,\mathcal{C}^p(\g)\rrbracket$$
for $p \geq 1$ In particular, a $2$-compatible Lie algebra is $2$-step nilpotent (or metabelian) if $\mathcal{C}^2(\g)=0.$ In this case the Lie algebras $(V,\mu)$ and $(V,\varphi_2)$ are also $2$-step nilpotent Lie algebras and $\varphi_1$ have to satisfy
$$
\left\{
\begin{array}{l}
      \mu(X,\varphi_1(Y,Z))+\varphi_1(X,\mu(Y,Z))=0    \\
       \mu(X,\varphi_2(Y,Z))+\varphi_2(X,\mu(Y,Z))+\varphi_1(X,\varphi_1(Y,Z))=0 \\
        \varphi_2(X,\varphi_1(Y,Z))+\varphi_1(X,\varphi_2(Y,Z))=0   
\end{array}
\right.
$$
for any $X,Y,Z \in V.$ 

Let us note that these relations imply in particular that $\varphi_1$ and $\varphi_2$ are $2$-cocycles for the Chevalley-Eilenberg cohomology of $\mu$. But the converse is false.

\noindent{\bf Examples}
\begin{enumerate}
\item The $7$-dimensional $2$-compatible Lie algebra given by
$$\mu(e_1,e_{2i})=e_{2i+1}, \ i=1,2,3, \ \ \varphi_1(e_2,e_4)=e_5, \ \varphi_2(e_4,e_6)=e_7$$
is $2$-step nilpotent.
  \item In \cite{G.H2} the notion of characteristic sequence of a complex nilpotent Lie algebra is defined. Let $(V,\mu)$ a complex nilpotent Lie algebra and let be $c(\mu)=(c_1,c_2,\cdots,c_k,1)$ its characteristic sequence. This means that there exists $X_1 \in V$ such that $c(\mu)$ is the characteristic sequence of the nilpotent operator $ad_\mu X_1$. Let $\{X_2,\cdots X_n,X_1)$ the Jordan basis associated with this sequence. Then $\g=(V,\mu)$ is a compatible Lie algebra, $\g=(V,\mu_0,\varphi)$ with $\mu_0$ given by
  $\mu_0(X_1,X_i)=\mu(X_1,X_i)$ and $\varphi(X_1,X_i)=0$ , $\varphi(X_i,X_j)=\mu(X_i,X_j).$ We can also consider $\g$ as a $2$-compatible Lie algebra. In fact we have $\g=(V,\mu_1,\varphi_1,\varphi_2)$ with
  $$
  \left\{
  \begin{array}{l}
  \mu_1(X_1,X_i)=X_{i+1}, \ i=2, \cdots c_2-1, \\
   \varphi_1(X_1, X_i)=\mu(X_1,X_i), \ i=c_2,\cdots,n, \\ 
   \varphi_2(X_i,X_j)=\mu(X_i,X_j) , i,j \neq 1.
  \end{array}
  \right.
  $$
  
\end{enumerate}

\section{Contact Lie algebras.}

\subsection{Some recalls}

Let $\g$ be a $2p+1$-dimensional contact Lie algebra, that is provided with a linear form $\om \in \g^*$, the dual space of $\g$, satisfying $\om \we d\om^p \neq 0$. Recall that $d\om$ is the $2$-exterior form defined by $d\om(X,Y)=-\om[X,Y]$ where $[,]$ is the Lie bracket of $\g$. The problem to classify up to an isomorphism the class of these algebras in a given dimension is, in the general form, still open. However, let us recall some essential results, but this list is not exhaustive.

$\bullet$ A semi-simple Lie algebra is a contact Lie algebra if and only if its rank is equal to $1$ \cite{Go}.

$\bullet$ The classification of  $7$-dimensional contact nilpotent Lie algebras as double extensions \cite{Salg1}.

$\bullet$  The description of some invariants of contact Lie algebras such that the Pfaffian and the characterization of perfect contact Lie algebras \cite{Salg2}.

$\bullet$ The classification of some contact Lie subalgebras (the seaweed algebras) of reductive Lie algebras. Recall that a seaweed algebra is a subalgebra of a reductive Lie algebra $\frak{r}$ defined
by the intersection of two parabolic subalgebras of $\frak{r}$ whose sum is $\frak{r}$ \cite{Co}.
\medskip

If $\g$ is a contact Lie algebra, there is a (not uniq) basis  $\{\om_1,\om_2,\cdots,\om_{2p+1}\}$ of $\g^*$ such that the contact form $\om=\om_1$ satisfies
$$d\om_1=\om_2 \we \om_3 + \om_4 \we \om_5 + \cdots + \om_{2p} \we \om_{2p+1}.$$
Such a basis, and its dual basis $\{e_1,e_2,\cdots ,e_{2p+1}\}$ of $\g$,  will be called a Darboux basis.  Then the structural equations of $\g$ are
$$
\left\{
\begin{array}{l}
     d\om_1=\om_2 \we \om_3 + \om_4 \we \om_5 + \cdots + \om_{2p} \we \om_{2p+1}   \\
     d\om_i= \sum C_{jk}^i \om_j \we \om_k ,\ \ i \geq 2
\end{array} 
\right.
$$
For example, the $(2p+1)$-dimensional Heisenberg Lie algebra $\mathcal{H}_{2p+1}$ is a contact Lie algebra corresponding to $C_{jk}^i=0.$

 In such a basis, the vector $e_1$ plays a special role. It is the only vector satisfying
$$\om(X)=1, \ \ i(X)d\om=0$$
where $i(X)d\om (Y)=d\om(X,Y)$. This vector is called the Reeb vector.

 \subsection{On the ideals of a contact Lie algebra}  Assume  that  $\g=(V,\mu)$ is a contact Lie algebra, i.e. $\g=(\mu_0,\varphi_1,\varphi_2)$ a $2$-compatible Lie algebra with $\mu_0$ the Heisenberg bracket and $\{e_1,\cdots,e_{2p+1}$ a Darboux basis associated with the contact form $\om_1$.
 According to \cite{Go}, any contact Lie algebra of dimension greater or equal to  $5$ admits a nontrivial ideal. In dimension $3$, then  $sl(2)$ or $so(3)$ are contact Lie algebras, but any contact Lie algebra of dimesion greater or equal to $5$ is not simple. 
 
 \begin{proposition} Let $\g$ be a $2p+1$-dimensional contact Lie algebra, $p \geq 2$. Then $\g$ admits a non trivial ideal. If $I$ is an nontrivial ideal of  $\g$, then there exists a vector $U \in \g$ such that $e_1+U \in I$, where $e_1$ is the Reeb vector  associated with the contact form.
 \end{proposition}
 \pf Recall that the Reeb vector $R$ of a contact form $\om$ is given by $\omega(R)=1, \ d\om(R,X)=0$ for any $X \in \g$. Assume that $\om(I)=0$. Then for any $X \in I$ we have $\om(X)=0$ and $i(X)d\om (Y)=-\omega[X,Y]=0$ for any $Y \in \g$. Then $I$ is a subspace of the characteristic space of $\om$. Since this space is reduced to $0$, we have a contradiction. This implies that there is $X=R+U \in I$. 
 
 \medskip
 
 \noindent{\bf Example}.  Let $\g$ be the $5$-dimensional Lie algebra isomorphic to $so(3) \oplus \mathfrak{r}(2)$ where $\mathfrak{r}(2)$  is the non abelian solvable $2$-dimensional Lie algebra. Let us consider the Lie bracket of $\g$ given by
$$[e_1,e_2]=e_2, \ [e_1,e_3]=-e_3, \ [e_2,e_3]=e_1, \ [e_2,e_4]=e_2, \ [e_3,e_4]=-e_3, \ [e_4,e_5]=e_1+e_4.$$
If $\{\om_i\}$ is the dual basis, then $\om_1$ is a contact linear form, $e_1$ is its Reeb vector and $I=\K\{e_1+e_4,e_5\}$ is an Ideal of $\g$. Here we have $U=e_4$. Let us note that, in this case, $f$ is given by
$$f(e_2)=e_2, \ f(e_3)=-e_3$$
and its rank is $2$.

\subsection{Contact Lie algebras and $2$-compatible Lie algebras}

In \cite{G.R.contact} we have seen that any $2p+1$-dimensional contact Lie algebra is isomorphic to a quadratic deformation of $\mathcal{H}_{2p+1}$. This is equivalent to say that, in a Darboux basis, we can consider that $\g$ is a $2$-compatible Lie algebra
$\g=(V,\mu,\varphi_1,\varphi_2)$ where $\mu$ is the Lie bracket of $\mathcal{H}_{2p+1}$. If $\{e_1,\cdots,e_{2p+1}\}$ is a Darboux basis and if  $(C_{ij}^k)$ are the structure constants of $\g$ with respect this basis  then the bilinear maps $\varphi_1$ and $\varphi_2$ are given by 
 $$
\left\{
\begin{array}{l}
\ds
\varphi_1(e_i,e_j)=\sum_{k=2}^{2p+1}C_{ij}^ke_k,   \ \ 2 \leq i < j \leq 2p+1,\\ 
\ds
  \varphi_2(e_1,e_i)= \sum_{k \geq 2}C_{1i}^k e_k,  \ \ 2 \leq i \leq 2p+1,
\end{array} 
\right.  
$$
 the non defined products are considered to be zero. These bilinear maps must satisfy
 \begin{align}
\label{EQ2} \delta_\mu \varphi_1=0,\\
\label{EQ3} \varphi_1 \circ \varphi_1 + \delta_\mu \varphi_2=0,\\
\label{EQ4} \varphi_1 \circ \varphi_2 + \varphi_2 \circ \varphi_1=0,\\
\label{EQ5} \varphi_2 \circ \varphi_2=0.
 \end{align}

\medskip

\noindent{\bf Example: Dimension 3}
Assume that $\mu$ is the Lie bracket of $\mathcal{H}_3$:
$$\mu(e_2,e_3)=e_1.$$
 For any skew-symmetric bilinear map given by $\psi(e_i,e_j)=\sum x_{ij}^k e_k$, we have
$$\delta_\mu\psi(e_1,e_2,e_3)=(x_{12}^2+x_{13}^3)e_1$$
Assume that $\g$ is a contact $3$-dimensional Lie algebra and let $\g=(V,\mu,\varphi_1,\varphi_2)$ be the associated $2$-compatible Lie algebra. We have
$$
\left\{
\begin{array}{l}
\varphi_1(e_2,e_3)=C_{23}^2 e_2+C_{23}^3 e_3,\\
\varphi_2(e_1,e_2)=C_{12}^2e_2+C_{12}^3e_3, \ \ \varphi_2(e_1,e_3)=C_{13}^2e_2+C_{13}^3e_3.
\end{array}
\right.
$$
Recall that the non defined product are considered to be null. Equations (\ref{EQ2}) and (\ref{EQ5}) are satisfied.  
Since $\varphi_1(e_i,e_j) \in \K\{e_2,e_3\}$, we have $\varphi_1 \circ \varphi_1 (e_1,e_2,e_3)=0$. Thus   $\delta_\mu \varphi_2=0$. This is equivalent to 
$$C_{13}^3=-C_{12}^2.$$ Now, Equation (\ref{EQ4}) is equivalent to
$$
\begin{pmatrix}
  C_{12}^2   &   C_{13}^2 \\
  C_{12}^3   &  -C_{12}^2 \\
\end{pmatrix}
\begin{pmatrix}
      C_{23}^2    \\
     C_{23}^3 
\end{pmatrix}=0.$$
If $\varphi_2=0$, then   the triple $(V,\mu,\varphi_1)$ is a compatible contact Lie algebra. In this case the $3$-dimensional contact Lie algebra is a one dimensional extension of a $2$-dimensional symplectic Lie algebra $\g_1=(\K\{e_2,e_3\},\varphi_1)$ with symplectic form  $\theta=\om_2 \we \om_3.$ In this case, $\g$ is isomorphic to $\mathcal{H}_3$.

\noindent If $\varphi_2 \neq 0$, then

i) $\Delta =-(b_{12}^2)^2-b_{12}^3b_{13}^2 \neq 0$. This implies  $\varphi_1=0$ and $\g$ corresponds to a simple Lie algebra,

ii)  $\Delta=0$. The eigenvalues of the matrix $ \begin{pmatrix}
   C_{12}^2  &  C_{12}^3  \\
C_{13}^2   &  -C_{12}^2
\end{pmatrix}$ are $0$. 
If $f$ is the endomorphism of $\K\{e_2,e_3\}$ corresponding to this matrix, this endomorphism is nilpotent. Then $f=0$ or $f$ can be reduced to $f(e_2)=e_3.$
This gives the complete list of $3$-dimensional contact Lie algebras.



\bigskip

Let us resume the study of the general case. Let $\g$ be a $(2p+1)$-dimensional contact Lie algebra with $p \geq 2$ and $(V,\mu,\varphi_1,\varphi_2)$ its corresponding $2$-compatible structure. 
\subsubsection{Description of $\varphi_2$}

By construction Equation (\ref{EQ5}) is always satisfied and $(V,\varphi_2)$ is a solvable Lie algebra obtained by an extension by derivation of a $2p$-dimensional abelian Lie algebra. Let $V_2=\K\{e_2,\cdots,e_{2p+1}\}$ be this abelian Lie algebra. Let $f$ be the endomorphism of $V_2$  given by
$$f(e_i)=\varphi_2(e_1,e_i)=\sum_{k \geq 2}C_{1i}^k e_k, \ \ 2 \leq i \leq 2p+1.$$
Since $\varphi_1 \circ \varphi_1(e_1,e_i,e_j)=0$, Equation (\ref{EQ3}) implies  $\delta_\mu\varphi_2(e_1,e_i,e_j)=0.$  In particular 
$$
\begin{array}{l}
 \delta_\mu\varphi_2(e_1,e_{2i},e_{2j})=0 \Rightarrow -C_{1,2i}^{2j+1}+C_{1,2j}^{2i+1}=0,    \\
 \delta_\mu\varphi_2(e_1,e_{2i},e_{2j+1})=0 \Rightarrow C_{1,2i}^{2j}+C_{1,2j+1}^{2i+1}=0,    \\
  \delta_\mu\varphi_2(e_1,e_{2i+1},e_{2j+1})=0 \Rightarrow -C_{1,2i+1}^{2j}+C_{1,2j+1}^{2i}=0.    \\
\end{array}
$$
We deduce that the matrix of $f$ in the Darboux basis $\{e_2,\cdots,e_{2p+1}\}$ is of the type
     $$ F=
     \begin{pmatrix}
    A_{11}  & \widetilde{A_{21}}  & \widetilde{A_{31}} & \widetilde{A_{41}} & \cdots & \widetilde{A_{p1}}  \\
     A_{21} & A_{22} & \widetilde{A_{32}} &  \widetilde{A_{42}} &\cdots & \widetilde{A_{p2}}  \\
      A_{31} & A_{32} & A_{33} & \widetilde{A_{43}}& \cdots & \widetilde{A_{p3}}  \\
      \cdots &\cdots &\cdots &\cdots & \cdots &\cdots  \\
      A_{p1} & A_{p2} & A_{p3} & A_{p4} &\cdots & A_{pp}
\end{pmatrix}
$$
where the matrices $A_{ij}$ are of order $2$,  $tr A_{ii}=0$ for $i=1,\cdots,p$ and the matrices $\widetilde{A_{ij}}$ are defined as follows:
if $A=
\begin{pmatrix}
     a & b   \\
     c & d
\end{pmatrix}$, then  $\widetilde{A}=
\begin{pmatrix}
     -d & b   \\
     c & -a
\end{pmatrix}$, that it the opposite transposed of the cofactors matrix associated with $A$. In particular we have
$$ A \cdot \widetilde{A}=-\det A.Id.$$

\medskip
\begin{proposition}
Let $\tilde{f}$ be the endomorphism of $V$ defined by $\tilde{f}(e_1)=0$ and $\tilde{f}(e_i)=f(e_i)$ for $i \geq 2$. Then $\tilde{f}\in \mathcal{D}er(\mu)$.
\end{proposition}
\begin{proof}
In fact $\delta_\mu\varphi_2(e_1,e_i,e_j)=0$ is equivalent to $\mu(f(e_i),e_j)+\mu(e_i,f(e_j))=0$.
 Then for $i,j \geq 1$, $\mu(\tilde{f}(e_i),e_j)+\mu(e_i,\tilde{f}(e_j))=0.$ 
By hypothesis $\tilde{f}e_1)=0.$ This gives $\tilde{f}(\mu(e_i,e_j))=0$ and $\tilde{f}$ is a derivation of the Heisenberg algebra which is trivial on the center. 
\end{proof}
This remark also makes it possible to find the $F$ matrix of $f$ 

\noindent{\textbf{Particular case}: $\delta_\mu\varphi_2=0$ }
Assume that $\delta_\mu\varphi_2=0$. This implies
\begin{proposition}\label{fi2}
If $p\geq 2$, then $\delta_\mu\varphi_2=0$ is equivalent to $\varphi_2=0$.
\end{proposition}
\begin{proof}
We have
$$\delta_\mu\varphi_2(e_{2i},e_{2i+1},e_k)=\varphi_2(e_1,e_k)$$
as soon as the vectors $e_{2i},e_{2i+1},e_k$ are linearly independent. Since $\delta_\mu\varphi_2=0$, then $\varphi_2(e_1,e_k)=f(e_k)0$ for any $k >1$ and $\varphi_2=0$.
\end{proof}
If $\delta_\mu\varphi_2=0$, Equation (\ref{EQ2}) implies $\varphi_1 \circ \varphi_1=0$ and $\g=(V,\varphi_1,\varphi_2)=(V,\varphi_1)$ is a compatible Lie algebra. Since $\varphi_1(e_1,x)=0$ for any $x \in V$, then $\g$ is a direct product $\K{e_1} \times \g_{2p}$ where $\g_ {2p}$ is a $2p$-dimensional Lie algebra on $\K\{e_2,\cdots,e_{2p+1}\}$ whose Lie bracket is given by $\varphi_1$. 

\begin{proposition} Assume that $\delta_\mu\varphi_2=0$. Then $\g$ is a compatible contact Lie algebra $\g=(V,\mu,\varphi_1)$ which is a one dimensional central extension $\g=\K\{e_1\} \oplus \g_{2p}$ of a $2p$-dimensional symplectic Lie algebra.
\end{proposition} 
\begin{proof}
In fact, the condition  $\delta_\mu\varphi_1=0$ is equivalent to say that the exterior $2$-form
 $$\theta=\om_2 \we \om_3 +\cdots + \om_{2p} \we \om_{2p+1}$$
 is a symplectic form on $\g_{2p}$.
\end{proof}

\medskip

\noindent \textbf{Remark: Frobeniusian extension.} The contact form on $\g$ is associated with a central extension of a symplectic sub-algebra if and only if $\K\{e_1\}$ is an ideal of $\g$ where $e_1$ is the Reeb vector of the contact form.   When the symplectic form is exact, there is a different type of symplectic Lie algebra extension \cite{RV}. Our focus is on a Frobeniusian Lie algebra with dimension $2p$ and a derivation $D$ of this algebra. 
\begin{proposition}
Let $\h$ be a $2p$-dimensional frobeniusian Lie algebra ($p \geq 2$) and $D$ a not inner derivation of $\h$. Then  the extension $\g$ of $\h$ by the derivation $D$ which is is defined by $\g=\h \oplus \K Z$ with $[Z,X]=D(X)$ for any $X \in \h$ is a contact Lie algebra.
\end{proposition}
 For example, if $\h$ is the $4$-dimensional frobeniusian model, that is given by
$$[Y_1,Y_2]=[Y_3,Y_4]=Y_1, \ [Y_2,Y_i]=-1/2Y_i,\ i=3,4$$
Let $D$ be the derivation given by 
$$D(Y_3)=a_1Y_3+a_2Y_4, \ D(Y_4)=a_3Y_3-a_1Y_4$$
It defines the extension $\g=\h \oplus \K\{Y_5\}$:
$$[Y_1,Y_2]=[Y_3,Y_4]=Y_1, \ [Y_2,Y_i]=-1/2Y_i,\ i=3,4, [Y_5,Y_3]=a_1Y_3+a_2Y_4, \ [Y_5,Y_4]=a_3Y_3-a_1Y_4.$$
The structural constants of $\g$ are
$$
\left\{
\begin{array}{l}
     d\alpha_1=\alpha_1 \we \alpha_2+ \alpha_3 \we \alpha_4    \\
     d\alpha_2=0\\
      d\alpha_3=-\frac{1}{2} \alpha_2 \we \alpha_3 -\alpha_1 \we (a_1 \alpha_3 +a_3 \alpha_4)\\
 d\alpha_4=-\frac{1}{2} \alpha_2 \we \alpha_4 +\alpha_1 \we (a_2 \alpha_4 -a \alpha_5)\\
 d\alpha_5=0
     \end{array} 
\right.
$$
The linear form $\om=\alpha_2+\alpha_5)$ is a contact form on $\g$. Let us note that, in this case , $f=D$ and $\g$ admits a $4$-dimensional ideal. 

Conversely, let $\g$ be a $2p+1$-dimensional contact Lie algebra admitting a $2p$-dimensional ideal $I$. If $e_1 \notin I$ and if $f$ is a derivation of $I$, then $\g$ is an extension of the frobeniusian algebra $I$ by the derivation $f$. 

\subsubsection{Description of $\varphi_1$}

Let us now study the relation,
$$ \delta_{\varphi_2} \varphi_1=\varphi_1 \circ \varphi_2+\varphi_2 \circ \varphi_1=0.$$
Recall that $\varphi_2(e_1,e_i)=f(e_i)$ and $\varphi_2(e_i,e_j)=0, \ i,j \neq 1.$ This gives
$$ \delta_{\varphi_2} \varphi_1(e_i,e_j,e_k)=0, \ i,j,k \neq 1$$
and 
$$ \delta_{\varphi_2} \varphi_1(e_1,e_i,e_j)=\varphi_2(\varphi_1(e_i,e_j),e_1)+\varphi_1(\varphi_2(e_1,e_i),e_j)+\varphi_1(\varphi_2(e_j,e_1),e_i)$$
that is
$$\delta_{\varphi_2} \varphi_1(e_1,e_i,e_j)=-f(\varphi_1(e_i,e_j))+\varphi_1(f(e_i),e_j)+\varphi_1(e_i,f(e_j)).$$
\begin{proposition}
The relation $ \delta_{\varphi_2} \varphi_1=0$ is equivalent to $f \in \mathcal{D}er (\varphi_1).$
\end{proposition}

\subsection{The algebra $(V,\varphi_1)$}
\subsubsection{On the structure of this algebra }

Recall that $\varphi_1(e_1,x)=0$ for any $x \in V$. We have
$$\delta_\mu\varphi_2(e_{2i},e_{2i+1},e_k)=\varphi_2(e_1,e_k).$$
We deduce (Equation (\ref{EQ2}) that for any $k$, 
$$\varphi_1 \circ \varphi_1(e_{2i},e_{2i+1},e_k)=-f(e_k).$$
So, we obtain
$$\varphi_1(\varphi_1 \circ \varphi_1(e_{2i},e_{2i+1},e_k),e_l)=-\varphi_1(f(e_k),e_l).$$
We have also
$$f(\varphi_1(e_k,e_l))=-\varphi_1 \circ \varphi_1(e_{2i},e_{2i+1},\varphi_1(e_k,e_l)).$$
Since $f \in \mathcal{D}er(\varphi_1)$, 
$$
-\varphi_1 \circ \varphi_1(e_{2i},e_{2i+1},\varphi_1(e_k,e_l))+\varphi_1(\varphi_1 \circ \varphi_1(e_{2i},e_{2i+1},e_k),e_l)-\varphi_1(\varphi_1 \circ \varphi_1(e_{2i},e_{2i+1},e_l),e_k)=0.$$
Let’s notice that if in the triplet $(i,j,k)$ there are no consecutive integers, then $\delta_\mu \varphi_2 (e_i,e_j,e_k)=0$ implying $\varphi_1 \circ \varphi_1(e_i,e_j,e_k)=0.$
We deduce
\begin{proposition}
Let $\g=(V,\mu,\varphi_1,\varphi_2)$ be a $2$-compatible Lie algebra associated with a quadratic deformation of $\mathcal{H}_{2p+1}$. Then $\varphi_1$ is a multiplication on $V$ satisfying the relation of degree $3$:
$$
-\varphi_1 \circ \varphi_1(x,y,\varphi_1(z,v))+\varphi_1(\varphi_1 \circ \varphi_1(x,y,z),v)-\varphi_1(\varphi_1 \circ \varphi_1(x,y,v),z)=0.$$
$x,y,z,v \in V_2.$
\end{proposition}

We can define another structure on this algebra. Let us assume that
$\varphi_1\neq 0$ and $\varphi_2 \neq 0$. The relation between these two bilinear maps is
$$\label{fi}
\varphi_1 \circ \varphi _1= -\delta_{\mu_0}\varphi_2.$$ We can however interpret the above relation as a generalization of the Jacobi relation. Let us recall that generalizations have already been made, for example the notion of type Lie algebra \cite{Mak}, or the notion of Hom-Lie algebra.  Let us recall that $H^*(\mu,\mu)$ denotes the Chevalley-Eilenberg complex of the Lie algebra whose Lie bracket is $\mu$. Let us consider the Schur multiplier 
$$M(\mu_0) = Z^2(\mu_0,\mu_0)/B^2(\mu_0,\mu_0)$$
associated with the Chevalley-Eilenberg complex of the Heisenberg algebra $(V,\mu_0)$ (see \cite{Ba} for a description of this notion). Since $\varphi_1(e_1,X)=0$ for any $X \in V$, we can consider that $\varphi_1 \in M(\mu_0)$ and $\varphi_1 \neq 0$ in this space. Here we denote by the same symbol the cohomology class of a cocycle and the cocycle itself.
\begin{definition}
Let $\g_0=(\K^n,\mu_0)$ be a finite dimension $\K$-Lie algebra. Let  $\varphi$ be a skew-symmetric bilinear map on $\K^n$. We say the the couple $\g_1=(\K^n,\varphi)$ is a Lie algebra modulo $\mu_0$ if\begin{enumerate}
  \item $\varphi \in M(\mu_0)$
  \item $\varphi \circ \varphi =0$ mod $B^3(\mu_0,\mu_0)$
\end{enumerate}
\end{definition}
For example, if $\g=(\K^n,\mu_0)$ is the abelian Lie algebra, then $\g_1=(\K^n,\varphi)$ is a Lie algebra modulo $\mu_0$ iff $\g_1$ is a Lie algebra. 
\begin{proposition}
Let $\g=(\K^{2p+1},\mu_0,\varphi_1,\varphi_2)$ a contact Lie algebra. Then
\begin{enumerate}
  \item $\varphi_1 \in M(\mu_0)$
  \item $\g_1=(\K^{2p+1},\varphi_1)$ is a Lie algebra modulo $\mu_0$ the Lie bracket of the Heisenberg Lie algebra.
\end{enumerate}
\end{proposition}
In fact, the relation (\ref{fi}) shows that $\varphi_1 \circ \varphi _1 \in B^3(\mu_0,\mu_0)$ then its class in $H^3(\mu_0,\mu_0)$ is trivial.

\subsection{On the derivation $f$ of the algebra $(V,\varphi_1)$}
\subsubsection{The Jordan-Chevalley decomposition} We assume now that $\K$ is algebraically closed. 
We have seen that the condition $\delta_{\varphi_2}\varphi_1=0$ was equivalent to $f \in \mathcal{D}er(\varphi_1)$ where $f$ is defined on $V_2=\K\{e_2,\cdots,e_{2p}\}$ by $f(e_i)=\varphi_2(e_1,e_i)$. Let $f=f_s +f_n$ be the Jordan-Chevalley decomposition of the linear operator $f$ on $V_2$. We shall prove that  $f_s$ and $f_n$ are also derivations of the algebra $(V_2,\varphi_1)$. In fact, let us assume that $f$ is not nilpotent. Let  $(\lambda_1, \cdots,\lambda_r)$ the set of eigenvalues of $f$ with $\lambda_i \neq\lambda_j$ when $i \neq j$. If $C_{\lambda_i}$ is the characteristic subspace of $V_2$ associated with $\lambda_i$, that is $C_{\lambda_i} = \ker (f-\lambda_i Id)^{s_i}$ where $s_i$ is the multiplicity of the root $\lambda_i$,  then $V_2=C_{\lambda_1} \oplus \cdots \oplus C_{\lambda_r}.$ 
The subspaces  $C_{\lambda_i}$  are invariant by $f$ and the restriction of $f_s$ on  $C_{\lambda_i}$ (more precisely the induced endomorphism by $f_s$ on $C_{\lambda_i}$) is $\lambda_i Id$. Consider $(X_i,X_j) \in C_{\lambda_i} \times C_{\lambda_j}$. Then
$$f(\varphi_1(X_i,X_j))=\varphi_1(f(X_i),X_j))+\varphi_1(X_i,f(X_j)).$$
This implies
$$
\begin{array}{ll}
 (f-(\lambda_i+\lambda_j)Id)(\varphi_1(X_i,X_j)     &= \varphi_1(f(X_i),X_j))+\varphi_1(X_i,f(X_j))-  (\lambda_i+\lambda_j)\varphi_1(X_i,X_j)  \\
      & =   \varphi_1((f-\lambda_i)(X_i),X_j))+\varphi_1(X_i,(f-\lambda_j)(X_j))\\
\end{array}
$$
By induction, we prove the following lemma
\begin{lemma}
$$ (f-(\lambda_i+\lambda_j)Id)^n(\varphi_1(X_i,X_j) =\sum_{k=0}^n \binom{n}{k}\varphi_1((f-\lambda_iId)^k(X_i),(f-\lambda_jId)^{n-k}(Y_j))$$
\end{lemma}
Assume now that $n$ is sufficiently large. Since $(X_i,X_j) \in C_{\lambda_i} \times C_{\lambda_j}$, then $(f-\lambda_iId)^k(X_i)=0$ or  $(f-\lambda_jId)^{n-k}(Y_j))=0$. We deduce
$$\varphi_1(X_i,X_j) \in C_{\lambda_i+\lambda_j}$$
and $\varphi_1(X_i,X_j)$ if $\lambda_i+\lambda_j$ is not in the spectrum of $f$. 
This implies
$$\begin{array}{ll}
    f_s(\varphi_1(X_i,X_j))  &  =(\lambda_i+\lambda_j)\varphi_1(X_i,X_j))  \\
      & =  \lambda_i\varphi_1(X_i,X_j) +\lambda_j\varphi_1(X_i,X_j)\\
      & = \varphi_1(f_s(X_i),X_j)+ \varphi_1(X_i,f_s(X_j))\\
\end{array}
$$ and $f_s$ is a derivation of $\varphi_1$

\begin{proposition}
If $f=f_s+f_n$ is the Jordan-Chevalley decomposition of $f$, then $f$, $f_s$ and $f_n$ are derivations of the algebra $(V,\varphi_1)$.
\end{proposition}

Let  $\{v_2,\cdots,v_{2p},v_{2p+1}\}$ be a basis   of eigenvectors of $f_s$ associated with the eigenvalues $\{\lambda_2,\cdots,\lambda_{2p},\lambda_{2p+1}\}$. Thus we have
$$f_s(\varphi_1(v_i,v_j))=(\lambda_i+\lambda_j)\varphi_1(v_i,v_j).$$
If $\lambda_i+\lambda_j$ is not an eigenvalue, then $\varphi_1(v_i,v_j)=0$ and $\varphi_1(C_{\lambda_i} \times C_{\lambda_j})=0.$ In the other case,
$$f_s(\varphi_1(v_i,(\varphi_1(v_i,v_j)))=(2\lambda_i+\lambda_j)\varphi_1(v_i,v_j)$$
Let us put $ad_{\varphi_1}(X)(Y)=\varphi_1(X,Y)$. This operator on the skewsymmetric nonassociative algebra $(V\varphi_1)$ corresponds to the left translation. The previous identity is written
$$f_s(ad_{\varphi_1}^2(v_i)(v_j))=(2\lambda_i+\lambda_j)\ad_{\varphi_1}^2(v_i)(v_j))$$
ad we have for any $k$
$$f_s(ad_{\varphi_1}^k(v_i)(v_j))=(k\lambda_i+\lambda_j)\ad_{\varphi_1}^k(v_i)(v_j)).$$
This shows that $k\lambda_i+\lambda_j$ is an eigenvalue for any $k$. Since this is impossible, there is $k$ such that  $\ad_{\varphi_1}^k(v_i)=0.$ Then for any $X \in V_2$, there exists $k$ such that 
$$\varphi_1(v_i,\varphi(v_i, \cdots \varphi_1(v_i,X)\cdots ))=0$$
for any eigenvectors $v_i$ that is $ad_{\varphi_1}(v_i)^k=0.$ Similarly, there is $k \in \N$ such as the product of order $k$ $ad_{\varphi_1}(v_{i_1})\circ \cdots ad_{\varphi_1}(v_{i_k})=0.$
Since $\{v_2,\cdots,v_{2p}\}$ is a basis of $V_2$, for any $Y \in V$, there exists $k$ such that $ad_{\varphi_1}(Y)^k=0$ and 
\begin{proposition}
The non associative algebra $(V,\varphi_1)$ is nilpotent.
\end{proposition}
Let us assume that $f$ is a non nilpotent singular derivation. If $s$ is the order of the null eigenvalue of $f_s$, for any vector $v \in \ker f_s$, we have for any eigenvector $v_i$ associated with the nonnull eigenvalue $\lambda_i$
$$f_s(\varphi_1(v,v_i))=\lambda_i\varphi_1(v,v_i)$$
and $\varphi_1(v,v_i)$ is zero or an eigenvector associated to $\lambda_i$. In this case
$$f_s(ad_{\varphi_1}^k(v)(v_i))=k\lambda_i\ad_{\varphi_1}^k(v)(v_i)$$
for any $k$ implying that there exists $k$ such as $\varphi_1^k(v,v_i)=0$.

 \subsubsection{The derivation $f$ is singular} We assume that $\dim \g \geq 5$. In this case $\g$ admits a non trivial ideals.  
Let $I$ be a non trivial ideal of $\g$ and $\{e_1,\cdots,e_{2p+1}$ a Darboux basis of $\g$. 

i) Assume, in a first time, that $I=\K\{e_1\}$. Then $f(e_i)=\varphi_2(e_1,e_i)=\mu(e_1,e_i)=0$, that is $f=0$. In this case, the contact Lie algebra is a central extension of a $2p$-dimensional symplectic Lie algebra. Conversely, if $f=0$, then $\K\{e_1\}$ is an ideal of $\g$.

\medskip 

ii) More generally, assume that $e_1 \in I$.  Then for any $X \in \g$, we have
$$\mu(e_1,X)=\varphi_2(e_1,X)=f(X) \in I.$$
If the rank of $f$ is $2p$, then $I=\g$. But $I$ is a proper ideal of $\g$. This implies that
$${\rm rk}(f) < 2p$$
and $f$ is a singular endomorphism of $V_2=\K\{e_2,\cdots,e_{2p+1}\}$. .

\medskip
iii)   Now, assume that $e_1 \notin I$. There is $U \neq 0$ such that  $e_1+U \in I$.  We have 
$$\mu(e_1,e_1+U)=\mu(e_1,U)=f(U) \in I.$$ But we have also $\mu(e_1,f(U))=f^2(U) \in I$ and more generally $f^k(U) \in I.$ Then $f(U)=0$ or there exists $k \geq 1$ with $f^k(U)=a_1f(U)+\cdots +a_{k-1}f^{k-1}(U).$ In the latter case, assume that $f$ is nonsingular. Then the last identity gives 
$ f^{k-1}(U)=a_1U+\cdots +a_{k-1}f^{k-2}(U)$ and $U \in I$. In this case $e_1 \in I$ and we have a contradiction. Then $f$ is singular.
\begin{proposition}
Assume that $\g=(V,\mu)$ is a $(2p+1)$-dimensional contact Lie algebra with $p \geq 2$. Then the linear map $f$ associated with $\varphi_2$ is singular.
\end{proposition}

\begin{corollary}
If $p\geq 2$, then the rank of $f_s$ satisfies
$${\rm rk}(f_s) \leq 2p-2.$$
\end{corollary} 
\pf In fact, we have seen that the semi-simple part $f_s$ of $f$ is also a derivation of the nonassociative algebra $(V,\varphi_1)$. But, if $\lambda$ is an eigenvalue of $f_s$, then $-\lambda$ is also an eigenvalue. Thus if $0$ is an eigenvalue of $f$, it is also an eigenvalue of $f_s$ and its multiplicity is greater or equal to $2$.

\medskip

Let us return to the previous example 
$$[e_1,e_2]=e_2, \ [e_1,e_3]=-e_3, \ [e_2,e_3]=e_1, \ [e_2,e_4]=e_2, \ [e_3,e_4]=-e_3, \ [e_4,e_5]=e_1+e_4.$$ If $I=\K\{e_1+e_4,e_5\}$ is an ideal, there are other ideals and the contact form is non-zero on each of them. For example $I_2=\K\{e_1+e_4\}$ or $I_3=\{e_1,e_2,e_3\}$. In this last example, we see that the Reeb vector field $e_1$ belongs to the ideal $I_3$. In particular $\im f \subset I_3$. Let us note also that if $\om$ is a contact form on a Lie algebra $\g$, since the contact condition is an open condition, any  linear form close to $\om$ is also a contact form. This permits to assume that if $\g$ is a contact Lie algebra and $\om$ a contact form on $\g$, then the Reeb vector field associated with $\om$ is an element of an ideal of $\g$.


\section{On the classification of contact Lie algebras}It is certainly unrealistic to think that this approach can achieve the contact Lie algebras classification. It should be remembered that few general results are known for real or complex finite dimensional Lie algebra classification, except for the simple or semi-simple Lie algebras. For example, complex or real nilpotent Lie algebras are only classified up to the dimension $7$. Particular cases are also classified, but hoping for a  general classification in all dimensions today is unimaginable.  In this work, we are only interested in contact Lie algebras. Our approach regarding the classification of this family is based on the notion of quadratic deformations.  
We are therefore not going to propose here a general classification of contact Lie algebras, but rather describe what can be expected through this decomposition method based on quadratic deformations.

Other promising approaches have been developed \cite{Salg1, Co, Salg2} also making it possible to describe very precise classes. Seaweeds are an example. These algebras are based on arithmetic hypotheses, graphs, root systems for example.

\subsection{Particular cases}
\subsubsection{$\varphi_2=0$}
We have already studied this case. If $\varphi_2=0$, then the contact Lie algebra $\g$ is a central extension of one-codimensional symplectic Lie algebra.

\noindent\textbf{Example: Nilpotent case.}
Let $\g$ be a contact nilpotent Lie algebra. In this case the derivation $f$ of $\varphi_1$ is nilpotent, that is $f=f_n$. From \cite{Go}, the center of $\g$ is $1$-dimensional and from \cite{Salg1},  if $\om$ is a contact form, then its Reeb vector is a generator of the center of $\g$. We deduce
\begin{lemma}
If $\g$ is a contact nilpotent Lie algebra, then the derivation $f$ of $\varphi_1$ is trivial.
\end{lemma}
As a consequence, any contact nilpotent Lie algebra is a one dimensional central extension of a symplectic nilpotent Lie algebra. We find again the result of   \cite{Salg1} (Theorem 4.3).

\subsubsection{ $\varphi_1=0$} If $\varphi_1=0$, then $\delta_\mu \varphi_2=-\varphi_1 \circ \varphi_1=0.$ But this implies (Proposition (\ref{fi2})), when $p \geq 2$ that $\varphi_2=0$. In this case $\g$ is the Heisenberg Li algebra. If $\dim \g=3$, then we have seen that any $3$-dimensional contact Lie algebra is isomorphic to a $2$-compatible Lie algebra with $\varphi_1=0.$
\begin{proposition}
Let $\g$ be a $2p+1$-dimensional contact Lie algebra, $p \geq 2$, whose $2$-compatible decomposition corresponds to $\varphi_1=0$. Then $\g$ is the Heisenberg Lie algebra, that is $\mu=\mu_0$.
\end{proposition}

\subsection{ Case $\varphi_1\neq 0$ and $\varphi_2 \neq 0$.} 

We assume now that $\K=\C$. We have seen that $f$ is a derivation of $\varphi_1$. 
\begin{lemma}
Let $f$ be the endomorphism of $V_2=\K\{e_2,\cdots,e_{2p+1}\}$ associated with $\varphi_2$. If $\lambda$ is an eigenvalue of $f$, then $-\lambda$ is also an eigenvalue of $f$.
\end{lemma}
In fact, let $v$ be an eigenvector associated with $\lambda$, that is $f(v)=\lambda v$. We can choose a Darboux basis $\{e'_2,e'_3,\cdots,e'_{p+1}\}$ such that $v=e'_2$.  Then $\mu(e_1,e'_2)=\lambda e'_2$. In this basis, the matrix $F'$ of $f$ has the following form
$$F'=\begin{pmatrix}
      \lambda & a_{12} & -a_{42} & a_{32} & \cdots & a_{2p,2} & a_{2p-1,2} \\
     0 & -\lambda & 0 & 0 &  \cdots & 0 & 0 \\
     0 & a_{32}& a_{33} & a_{34} & \cdots &-a _{2p,4} & a_{2p-1,4} \\
     0 & a_{42}&a_{43} & a_{44} & \cdots & _{2p,3} & -a_{2p-1,3} \\
     \cdots \\
     0 & a_{2p-1,2}& a_{2p-1,3} & a_{2p-1,4} & \cdots &a_{2p-1,2p-1} & a_{2p-1,2p} \\
     0 & a_{2p,2}&a_{2p,3} & a_{2p,4} & \cdots &a _{2p,2p-1} & -a_{2p-1,2p-1} \\
     \end{pmatrix}
     $$
     and $-\lambda$ is an eigenvalue of $F$.

 Let $\h$ be the characteristic space of $f$ associated with the eigenvalue $0$. It is a subalgebra of $\g$. In fact, if $X,Y \in \h$,
$$f(\varphi_1(X,Y))=\varphi_1(f(X),Y)+\varphi_1(X,f(Y))=0$$
and $\varphi_1(X,Y) \in \h.$ Since $\varphi_2(X,Y)=0$, then $\mu(X,Y) \in \h$ and $\h$ is a Lie subalgebra of $\g$.
\begin{proposition}
$ \ker (f)$  is a Lie subalgebra of $\g$.
\end{proposition}
As a consequence, the characteristic space associated with the eigenvalue $0$ is also a subalgebra of $\g$. In fact if $\{v_2,\cdots,v_k\}$ is a Jordan block associated with the eigenvector $v_2$, that is $f(v_2)=0, f(v_i)=v_{v-1}$, then
$$f(\varphi_1(v_2,v_3))=0,\ f( \varphi_1(v_2,v_i))=\varphi_1(v_2,v_{i-1}), \ f(\varphi_1v_i,v_j))=\varphi_1(v_{i-1},v_j)+\varphi_1(v_i,v_{j-1}).$$

Let us denote by $\widetilde{\lambda_i}$ one of the two eigenvalues $\lambda_i$ or $-\lambda_i$. If 
$\mathcal{R}=\{\widetilde{\lambda_i}\neq 0\}$,  let us consider on $\mathcal{R}$ the linear system
$$\widetilde{\lambda_i}+\widetilde{\lambda_j}=\widetilde{\lambda_k}.$$
The rank of $\mathcal{R}$ is an invariant of the contact structure  and it determines a grading of $V_2$. 
\subsection{Classification }

\subsubsection{Classification modulo $GL(2p)$}

Let $\g$ be a contact Lie algebra, $\om$ a contact form  on $\g$ and $\{e_1,\cdots e_{2p+1}\}$ a Darboux basis associated with $\om$, $e_1$ being the Reeb field of $\om$. First, we can consider a  change of basis  that leaves $e_1$ and $V_2=\K\{e_2,\cdots,e_{2p+1}\}$ invariant. It is defined by an isomorphism $P$ of $\K^{2p+1}$ of which matrix in the Darboux basis is
$$P=
\begin{pmatrix}
1 & 0 \\
0 & P_1
\end{pmatrix}
$$
with $P_1 \in GL(2p)$. Thischange of basis preserves the $2$-compatible structure, but does not retain the reduced form of the contact form $\om$. The new base is not a Darboux basis of $\om$. 
\subsubsection{The Lie Group $R_p$ and the Lie algebra $\mathfrak{r}_p$}
After this isomorphism, we can reduce the form of contact to canonical form. The actions described below can be interpreted as the action of a Lie subgroup of $GL(2p+1)$. 
 Let $\Theta_p$ the skew-symmetric bilinear form on $\K^{2p}$ represented by $\om_2 \we \om_3 + \cdots+ \om_{2p} \we \om_{2p+1}$. It represented by the matrix
 $$
 \Theta_p=\begin{pmatrix}
    0  & 1 & 0 & 0 & \cdots & 0 & 0  \\
    -1 & 0 & 0 & 0 & \cdots & 0 & 0  \\
  0  & 0 & 0 & 1 & \cdots & 0 & 0  \\
   0 & 0 & -1 & 0 & \cdots & 0 & 0  \\
   \cdots \\
    0  & 0 & 0 & 0 & \cdots & 0 & 1  \\
    0 & 0 & 0 & 0 & \cdots &-1 & 0  \\
\end{pmatrix}
$$
\begin{proposition}
Let $R_p$ be the subgroup of $GL(2p,\K)$ given by
$$R_p=\{ A \in GL(2p,\K), 
 \,  
 ^{ t} \! A\Theta_p A=\Theta_p\}.$$
Then $R_p$ is a Lie subgroup of $GL(2p,\K)$.
\end{proposition}
In fact it is a closed subgroup of $GL(2p,\K)$. To compute its Lie algebra, we consider an infinitesimal deformation $A+\Delta$ of $A$. We obtain
$$^t(A+\Delta)\Theta_p (A+\Delta)=\Theta_p$$
and at the first order
$$^tA\Theta_p + \Theta_p A=0$$
\begin{proposition}
The Lie algebra $\mathfrak{r}_p$ of the Lie Group $R_p$ is the Lie subalgebra of $gl(2p,\K)$ given by
$$\mathfrak{r}_p = \{A \in gl(2p,\K),
\,  ^t \! A\Theta_p +\Theta_p A=0\}.$$
\end{proposition}
Let $A$ be in $\mathfrak{r}_p$. Let us write $A$ on the following form
$$A=
\begin{pmatrix}
    A_{11}  &  A_{12} & \cdots & A_{1p}  \\
 A_{21}  &  A_{22} & \cdots & A_{2p}  \\
\cdots \\
 A_{p1}  &  A_{p2} & \cdots & A_{pp}  \\
\end{pmatrix}
$$
where $A_{ij}$ is a $2 \times 2$ matrix. Then $A  \in \mathfrak{r}_p$ is equivalent to
\begin{enumerate}
  \item $ \textrm{tr}A_{ii}=0$, $i=1, \cdots p$ \\
  \item Let $ i \neq j$ and $A_{ij}=\begin{pmatrix}
    a  &  b  \\
     c &  d
\end{pmatrix}
$. Then $A_{ji}=\begin{pmatrix}
    -d  &  b  \\
    c  &  -a
\end{pmatrix}
$. We shall write $A_{ji}=\widetilde{A_{ij}}.$
\end{enumerate}
Then any element $A$ of $ \mathfrak{r}_p$ is written
$$A=
\begin{pmatrix}
    A_{11}  &  \widetilde{A_{21} }& \cdots & \widetilde{A_{p1} } \\
 A_{21}  &  A_{22} & \cdots & \widetilde{A_{p2}}  \\
\cdots \\
 A_{p1}  &  A_{p2} & \cdots & A_{pp}  \\
\end{pmatrix}
$$
with $ \textrm{tr}A_{ii}=0$.
We deduce 
$$\dim  \mathfrak{r}_p= p(2p+1).$$
Let's now study the structure of this algebra and start with $p=2$.  Any element of $\mathfrak{n}_2$ is of type
$$A=\begin{pmatrix}
    A_{11}  &   \widetilde{A_{21}}   \\
 A_{21}  &  A_{22}  \\
\end{pmatrix}
$$
Let $\mathfrak{h}_2$
 be the subspace of $\mathfrak{n}_2$ whose elements are the matrices
 $$\begin{pmatrix}
    A_{11}  &  0  \\
 0  &  A_{22}  \\
\end{pmatrix}
$$
Then $\mathfrak{h}_2$ is a semi-simple Lie subalgebra of $\mathfrak{n}_2$ isomorphic to $sl(2) \times sl(2)$. Let $\mathfrak{m}_2$ be the subspace of $\mathfrak{n}_2$ whose elements are
$$\begin{pmatrix}
   0 &   \widetilde{A_{21}}  \\
  A_{21} &  0  \\
\end{pmatrix}
$$
\begin{proposition}
The Lie algebra $\mathfrak{r}_2$ admits the decomposition
$$\mathfrak{r}_2=\mathfrak{h}_2 +\mathfrak{m}$$
with
$$\left[\mathfrak{h}_2 ,\mathfrak{h}_2 \right] \subset \mathfrak{h}_2 ,  \ \left[\mathfrak{m} ,\mathfrak{m}\right]\subset \mathfrak{h}_2, \ \left[\mathfrak{h}_2 ,\mathfrak{m}\right] \subset \mathfrak{m}$$
This decomposition corresponds to a $\mathbb{Z}_2$-grading of $\mathfrak{n}_2$.
Then if $H_2$ is the semi-simple Lie group associated with the Lie subalgebra $\mathfrak{h}_2$, the homogeneous space $R_2/H_2$ is a provided with a symmetric structure.
\end{proposition}
Finally, for the general case, we have the decomposition
$$\mathfrak{r}_p=
\begin{pmatrix}
    \mathfrak{h}_p  &   \mathfrak{m}_1 &  \mathfrak{m}_{p} & \cdots&  \mathfrak{m}_{(p^2-p-4)/2}  & \mathfrak{m}_{p(p-1)/2} \\
       \mathfrak{m}_1  &   \mathfrak{h}_p &  \mathfrak{m}_2& \cdots  &  \mathfrak{m}_{(p^2-p-8)/2}  & \mathfrak{m}_{(p^2-p-2)/2} \\ 
         \mathfrak{m}_{p}  &   \mathfrak{m}_{2} &  \mathfrak{h}_p &  &  \mathfrak{m}_{(p^2-p-14)/2}  &  \mathfrak{m}_{(p^2-p-6)/2}  \\
         \vdots  &   \vdots  & & \ddots & & \vdots \\
           \mathfrak{m}_{(p^2-p-4)/2} &   \mathfrak{m}_{(p^2-p-8)/2}  &  \mathfrak{m}_{(p^2-p-14)/2}  &  &   \mathfrak{h}_p  &  \mathfrak{m}_{p-1} \\
           \mathfrak{m}_{p(p-1)/2} &   \mathfrak{m}_{(p^2-p-2)/2}  &  \mathfrak{m}_{(p^2-p-6)/2}  & \cdots  &  \mathfrak{m}_{p-1} &\mathfrak{h}_p \\
\end{pmatrix}
$$
Let us consider in the group $\mathbb{Z}_2^{p-1}$ the subset $\Gamma$ whose elements are 
$$
\left\{
\begin{array}{l}
 \gamma_1=(0,0,\cdots,0,1),\gamma_2=(0,0,\cdots,1,0),\cdots,\gamma_{p-2}=(0,1,\cdots,0,0),\gamma_{p-1}=(1,0,\cdots,0,0) \\
\gamma_p=(0,0,\cdots,1,1),\gamma_{p+1}=(0,\cdots,1,1,0),\cdots, \gamma_{2p-3}=(1,1,0,\cdots,0) \\
 \gamma_{2p-2}=(0,\cdots,0,1,1,1),\gamma_{2p}=(0,\cdots,1,1,1,0),\cdots , \gamma_{3p-6}=(1,1,1,0,\cdots,0) \\
 \cdots \\
\gamma_{p(p-1)/2}=(1,1,\cdots,1,1) .\\
\end{array}
\right.
$$
Let's write now $\mathfrak{m}_{\gamma_i}$ instead of  $\mathfrak{m}_i.$
\begin{theorem}
The Lie algebra $\mathfrak{r}_p$ admits the following decomposition
$$\mathfrak{r}_p=\mathfrak{h}_p \bigoplus_{\gamma_i \in \Gamma} \mathfrak{m}_{\gamma_i}$$
with
\begin{enumerate}
  \item $\mathfrak{h}_p=sl(2)\times sl(2) \times \cdots \times sl(2)$ ($p$ times)
  \item $\left[\mathfrak{m}_{\gamma_i},\mathfrak{m}_{\gamma_i}\right]\subset \mathfrak{h}_p, \ \ \left[\mathfrak{m}_{\gamma_i}, \mathfrak{h}_p\right]\subset \mathfrak{m}_{\gamma_i}$
  \item $\left[\mathfrak{m}_{\gamma_i},\mathfrak{m}_{\gamma_j}\right]\subset \mathfrak{m}_{\gamma_i+\gamma_j}$ if $\gamma_i+\gamma_j \in \Gamma$, if not the bracket is $0.$
\end{enumerate}
If $H_p$ is the semi-simple Lie group $SL(2)^p$, then the homogeneous space $R_p/H_p$ is provided with a $\mathbb{Z}_2^{p-1}$-symmetric structure. 
\end{theorem}
The notion of $\mathbb{Z}_2^p$-symmetric spaces is given in \cite{Bah,G.R.gamma}. Let us note that the elements of $\Gamma$ are $p$-uples elements of $\mathbb{Z}_2$,  $(i_1,\cdots , i_p)$ such that if $i_k=1$ and $i_l=1$ with $l > k$, then $i_s=1$ for $s=k,k+1,\cdots,l.$

\begin{proposition}
Let $\g=(\K^{2p+1},\mu_0,\varphi_1,\varphi_2)$ be a contact Lie algebra. Let $f$ be the endomorphisme of $\K^{2p}$ associated with $\varphi_2$. Then the matrix $F$ of $f$ belongs to $\mathfrak{r}_p$.
\end{proposition}
Since $\det f=0$, we have reduced the matrix $F$ by taking the first column $C_1(F)=0$, implying that the second line 
$L_2(F)=0.$

The main problem consist to reduce $F$ by a change of Darboux basis. This is equivalent to say
\begin{proposition}
The problem to reduce the matrix $F$ by a change of Darboux basis, and then the problem of classification of contact Lie algebras, is equivalent to study the adjoint representation of the Lie group $R_p$ on its Lie algebra $\mathfrak{r}_p$
\end{proposition}
In fact, since $R_p$ is a matricial Lie group, that is a subgroup of $GL(2p,\K)$, then the adjoint representation  of $R_p$ on the Lie algebra $\mathfrak{r}_p$ is given by
$$Ad_g(A)=g A g ^{-1}, \ \ A \in \mathfrak{r}_p, g \in R_p.$$

\noindent \textbf{The subalgebra $\mathfrak{r_0}_p$}. 
As $0$ is an eigenvalue of the endomorphism $f$ associated with $\varphi_2$, we are led to focus on matrices of $\mathfrak{r}_p$ whose first column and second row are zero that is of the form 
$$A=\begin{pmatrix}
     0 & a_{12} & -a_{42} & a_{32} & \cdots & a_{2p,2} & a_{2p-1,2} \\
     0 &0& 0 & 0 &  \cdots & 0 & 0 \\
     0 & a_{32}& a_{33} & a_{34} & \cdots &-a _{2p,4} & a_{2p-1,4} \\
     0 & a_{42}&a_{43} & a_{44} & \cdots & _{2p,3} & -a_{2p-1,3} \\
     \cdots \\
     0 & a_{2p-1,2}& a_{2p-1,3} & a_{2p-1,4} & \cdots &a_{2p-1,2p-1} & a_{2p-1,2p} \\
     0 & a_{2p,2}&a_{2p,3} & a_{2p,4} & \cdots &a _{2p,2p-1} & -a_{2p-1,2p-1} \\
     \end{pmatrix}.$$
Let’s note down by $\mathfrak{r_0}_p$ the set of these matrices.
\begin{proposition}
The subspace $\mathfrak{r_0}_p$ of $\mathfrak{r}_p$ is a Lie subalgebra of $\mathfrak{r}_p$ and
$$\dim \mathfrak{r_0}_p=p(2p-1).$$
\end{proposition}

\subsection{Classification of diagonal contact Lie algebras}
Let $\g$ be a $2p+1$-dimensional contact Lie algebra and $\g=(\mu_0,\varphi_1,\varphi_2)$ its associated $2$-compatible structure (the Darboux basis is fixed). We say that $\g$ is diagonal if the endomorphism $f$ of $V_2$ associated with $\varphi_2$ is diagonal.

\textbf{Throughout this section, we make the following assumption: }
\begin{center}
\textit{The non-zero eigenvalues are simple and $0$ is of order $2$}
\end{center}

Let $\mathrm{Spec}^*(f)$ be the spectrum of $f$ without the null eigenvalue, the eigenvalues of this spectrum can be equal, that is  $\mathrm{Card}(\mathrm{Spec}^*(f)=\dim \g -\dim \mathrm{Ker f}$. Let $E$ be the $\K$-vector space generated by $\mathrm{Spec}^*(f)$  and $F$ be the kernel of the linear system:
$$\lambda_I+\lambda_j-\lambda_k =0, \ \ \lambda_i,\lambda_j,\lambda_k \in \mathrm{Spec}^*(f).$$ 
A root system $\mathcal{R}(\g)$ is a subset of $\mathrm{Spec}^*(f)$ which is a basis of $E/F$. 
\begin{definition}
The rank of $\g$ is $r(\g)=\frac{1}{2}\mathrm{Card}(\mathcal{R}).$
\end{definition}

\subsubsection{We assume $r(\g)=p-1$} that is the rank of $\g$ is maximal. In this case $\dim \mathrm{ker}(f)=2.$ Let's set $\mathrm{Spec}^*(f)=\{\lambda_4,\lambda_5,\cdots,\lambda_{2p},\lambda_{2p+1}\}$ with $\lambda_{2i+1}=-\lambda_{2i}.$ Then we have
$$f(e_{2i})=\lambda_{2i} e_{2i}, \ f(e_{2i+1})=-\lambda_{2i}e_{2i+1}, \ f(e_2)=f(e_2)=0.$$ Since $\f$ is a derivation of $\varphi_1$, we obtain
$$
\left\{
\begin{array}{l}
\varphi_1(e_2,e_3)=ae_2+be_3,\\
    \varphi_1(e_{2i},e_{2i+1})=a_ie_2+b_ie_3, \ \ i=2,\cdots,p   \\
    \varphi_1(e_2,e_{2i})=c_{2i}e_{2i}, \  \varphi_1(e_2,e_{2i+1})=c_{2i+1}e_{2i+1} \\
    \varphi_1(e_3,e_{2i})=d_{2i}e_{2i}, \  \varphi_1(e_3,e_{2i+1})=d_{2i+1}e_{2i+1} \\
    \varphi_1(e_k,e_l)=0 \ (k,l) \neq  (2i,2i+1).
\end{array}
\right.
$$
The condition $\delta_{\mu_0}\varphi_1=0$ gives
$$a_i=-d_{2i}-d_{2i+1}, \ \ b_i=c_{2i}+c_{2i+1}, \ \ i=2, \cdots,p.$$
Then, $\varphi_1$ is parametrized by $a_1,b_1$ and the other eigenvalues of $ad_{\varphi_1}e_2$ and of $ad_{\varphi_1}e_3$. Assume $p\geq 3$. Then  the condition $\varphi_1 \circ \varphi_1+ \delta_{\mu_0}\varphi_2=0$ is equivalent to the system
$$
\left\{
\begin{array}{l}
     a_i(a-a_i)=0   \\
     b_i(a-a_i)=0 \\
     a_i(b-b_i)=0 \\
     b_i(b-b_i)=0 \\
     a_jc_{2i}+b_jd_{2i}=-\lambda_{2i}, \ \ i \neq j \\
      ac_{2i}+bd_{2i}=-\lambda_{2i},\\
     a_jc_{2i+1}+b_jd_{2i+1}=\lambda_{2i}, \ \ i \neq j,\\
     ac_{2i+1}+bd_{2i+1}=\lambda_{2i}, \\
     d_{2i}+d_{2i+1}=-a_i,\\
     c_{2i}+c_{2i+1}=b_i.    
\end{array}
\right.
$$
Recall that, by hypothesis, $\lambda_i \neq 0$ for any $i \geq 2$. This shows in particular that $a$ and $b$ cannot be simultaneously null. Assume $a \neq 0$. By a change of Darboux basis given by $Y_2= e_2+\frac{b}{a}e_3$ and $Y_3=e_3$, we can assume that $b=0$ which implies $b=0$ for $i \geq 2$ and also $a_i=a$, We even considered that $a=1$.  These hypothesis  imply $b_i=0$ The previous system is therefore reduced to
$$
\left\{
\begin{array}{l}
    a_i(1-a_i)=0, \ i \geq 2 \\
    c_{2i}=-\lambda_{2i}, \ \ i \geq 2 \\
    c_{2i+1}=\lambda_{2i}, \ \ i \geq 2 \\
    a_jc_{2i}=-\lambda_{2i}, \ \ i \geq 2,\ j \neq i \\
    a_jc_{2i+1}=\lambda_{2i}, \ \ i \geq 2, \ j \neq i \\
d_{2i}+d_{2i+1}=-a_i.
\end{array}
\right.
$$

If $\dim \g \geq 7$, then $a_j =1$. If $\dim \g =5$, the fourth and fifth equations do not exist. We deduce from them that $a_2=0$ or $a_2=1$.
\begin{theorem}
Let $\g$ be a $(2p+1)$-dimensional diagonal contact Lie algebra with $p\geq 3$ of maximal rank. Then, up an isomorphism, the structural equations of $\g$ are: 
\begin{enumerate}[label=\Alph*),resume]
\item $
\left\{
\begin{array}{l}
   d\om_1= \om_2 \we \om_3 + \om_4 \we \om_5 + \cdots + \om_{2p} \we \om_{2p+1} \\
    d\om_2=  \om_2 \we \om_3 + \om_4 \we \om_5 + \cdots + \om_{2p} \we \om_{2p+1} \\
    d\om_3=0 \\
    d\om_4=\left(\lambda_4\om_1-\lambda_4\om_2+d_4 \om_3\right) \we \om_4 \\
    d\om_5=\big(-\lambda_4\om_1+\lambda_4\om_2+(-1-d_4) \om_3\big) \we \om_5 \\
    \cdots \\
    d\om_{2p}=(\lambda_p\om_1-\lambda_p\om_2+d_{2p} \om_3)\we \om_{2p} \\
    d\om_{2p+1}=\big(-\lambda_p\om_1+\lambda_p\om_2+(-1-d_{2p}) \om_3\big)\we \om_{2p+1} \\
\end{array}
\right.
$
\end{enumerate}
\end{theorem}

\noindent\textbf{Remark:} $\dim \g=5$ ( $p=2$). In dimension $5$ we have two models corresponding to $a_2=1$ or $a_2=0$, that is
$$ 
\left\{
\begin{array}{l}
d\om_1=\om_2 \we \om_3 + \om_4 \we \om_5,\\
d\om_2=\om_2 \we \om_3 + \om_4 \we \om_5,\\ 
d\om_3=0,\\
d\om_4= (\lambda_4 \om_1 \we \om_4-\lambda_4 \om_2 \we \om_4 + (-1-d_4)\om_3 \we \om_4,\\
d\om_5=(\lambda_4 \om_1 \we \om_6+ \lambda_4 \om_2 \we \om_5)+(-1-d_4)\om_3 \we \om_5
\end{array}
\right.
, $$
$$
\left\{
\begin{array}{l}
d\om_1=\om_2 \we \om_3 + \om_4 \we \om_5,\\
d\om_2=\om_2 \we \om_3 \\ 
d\om_3=0,\\
d\om_4= (\lambda_4 \om_1 \we \om_4-\lambda_4 \om_2 \we \om_4 + (-1-d_4)\om_3 \we \om_4,\\
d\om_5=(\lambda_4 \om_1 \we \om_6+ \lambda_4 \om_2 \we \om_5)-d_4\om_3 \we \om_5
\end{array}
\right.
.$$

\medskip

\subsubsection{We assume that $r(\g)=p-2$} That is there exist a linear relation $\lambda_i+\lambda_j=\lambda_k$.  The notations are the same as above. Since there is no order for the eigenvalues of $f$, we can assume that the only connection between the roots is:
$$\lambda_4+\lambda_6=\lambda_8$$
with the associated relations
$$
 \lambda_6-\lambda_8=-\lambda_4, \ \lambda_4-\lambda_8=-\lambda_6,\ -\lambda_4+\lambda_8=\lambda_6, \ -\lambda_6+\lambda_8=\lambda_4, 
 -\lambda_4-\lambda_6=-\lambda_8.$$
The structure constants are identical to the previous case with the following exception
$$
\left\{
\begin{array}{l}
\varphi_1(e_4,e_6)=\alpha_1 e_8,   \ \varphi_1(e_5,e_8)=\alpha_2 e_6, \ \varphi_1(e_7,e_8)=\alpha_3e_4.\\
\varphi_1(e_5,e_7)=\beta_1 e_9, \ \varphi_1(e_4,e_9)=\beta_2 e_7, \ \varphi_1(e_6,e_9)=\beta_3e_5, 
\end{array}
\right.
$$
The Jacobi conditions imply $a_2(1-a_3)=a_3(1-a_3)-a_4(1-a_4)=0.$ As we want to determine these classes of algebras up to isomorphism, we can only consider the following cases
$$
\begin{array}{ll}
i) & a_2=a_3=a_4=1\\
ii) & a_2=a_3=1,a_4=0\\
iii) & a_2=1, a_3=a_4=0\\
iv) & a_2=a_3=a_4=0.
\end{array}
$$

\noindent \textit{Case i)}. We obtain two cases
$$(\alpha_1,\epsilon_2,\epsilon_3,\beta_1,\beta_2,\beta3)=(\alpha_1,0,0,0,\beta_2,\alpha_1+\beta_2) \ \ \mathrm{or} \ \ (0,\alpha2,\alpha_3,\alpha-\alpha_2,0,0).$$
\begin{theorem}
Let $\g$ be a $(2p+1)$-dimensional diagonal contact Lie algebra with $p\geq 4$ and $r(\g)=p-2, \ p \geq 4$. Then, up an isomorphism, the structural equations of $\g$ are: 
\begin{enumerate}[label=\Alph*),resume]
\item $
\left\{
\begin{array}{l}
   d\om_1= \om_2 \we \om_3 + \om_4 \we \om_5 + \cdots + \om_{2p} \we \om_{2p+1} \\
    d\om_2=  \om_2 \we \om_3 + \om_4 \we \om_5 + \cdots + \om_{2p} \we \om_{2p+1} \\
    d\om_3=0 \\
    d\om_4=\left(\lambda_4\om_1-\lambda_4\om_2+d_4 \om_3\right) \we \om_4 \\
    d\om_5=\big(-\lambda_4\om_1+\lambda_4\om_2+(-1-d_4) \om_3\big) \we \om_5 +(\alpha_1+\beta_2) \om_6 \we \om_9\\
    d\om_6=\left(\lambda_6\om_1-\lambda_6\om_2+d_6 \om_3\right) \we \om_6 \\
    d\om_7=\big(-\lambda_6\om_1+\lambda_6\om_2+(-1-d_6) \om_3\big) \we \om_7 +\beta_2 \om_4 \we \om_9\\
    d\om_8=\left((\lambda_4+\lambda_6)(\om_1-\om_2)+d_8 \om_3\right) \we \om_8 + \alpha_1 \om_4 \we \om_6\\
    d\om_9=\big(-(\lambda_4-\lambda_6)(\om_1-\om_2)+(-1-d_8) \om_3\big) \we \om_9 \\
    \cdots \\
    d\om_{2p}=(\lambda_p(\om_1-\om_2)+d_{2p} \om_3)\we \om_{2p} \\
    d\om_{2p+1}=\big(-\lambda_p(\om_1-\om_2)+(-1-d_{2p}) \om_3\big)\we \om_{2p+1} \\
    \mathrm{with} \ d_4+d_6-d_8=0 
\end{array}
\right.
$
\item $
\left\{
\begin{array}{l}
   d\om_1= \om_2 \we \om_3 + \om_4 \we \om_5 + \cdots + \om_{2p} \we \om_{2p+1} \\
    d\om_2=  \om_2 \we \om_3 + \om_4 \we \om_5 + \cdots + \om_{2p} \we \om_{2p+1} \\
    d\om_3=0 \\
    d\om_4=\left(\lambda_4\om_1-\lambda_4\om_2+d_4 \om_3\right) \we \om_4+\alpha_3 \om_7 \we \om_8 \\
    d\om_5=\big(-\lambda_4\om_1+\lambda_4\om_2+(-1-d_4) \om_3\big) \we \om_5\\
    d\om_6=\left(\lambda_6\om_1-\lambda_6\om_2+d_6 \om_3\right) \we \om_6+ \alpha_2 \om_5 \we \om_8 \\
    d\om_7=\big(-\lambda_6\om_1+\lambda_6\om_2+(-1-d_6) \om_3\big) \we \om_7 \\
    d\om_8=\left((\lambda_4+\lambda_6)(\om_1\om_2)+d_8 \om_3\right) \we \om_8 \\
    d\om_9=\big(-(\lambda_4-\lambda_6)(\om_1+\om_2)+(-1-d_8) \om_3\big) \we \om_9+ (\alpha_3-\alpha_2) \om_5 \we \om_7 \\
    \cdots \\
    d\om_{2p}=(\lambda_p\om_1-\lambda_p\om_2+d_{2p} \om_3)\we \om_{2p} \\
    d\om_{2p+1}=\big(-\lambda_p\om_1+\lambda_p\om_2+(-1-d_{2p}) \om_3\big)\we \om_{2p+1} \\
    \mathrm{with} \  d_4+d_6-d_8+1=0.
\end{array}
\right.
$
\end{enumerate}
All these algebras are extensions of frobenusian Lie algebras
\end{theorem}

\noindent \textit{Case ii), iii), iv)}. These cases lead to a contradiction. In fact the Jacobi conditions imply that  one of the eigenvalues is null.
From the hypothesis, this  is impossible. 

\begin{corollary}
Let $\g$ be a $(2p+1)$-dimensional diagonal contact Lie algebra with $p\geq 4$ of maximal rank. Then $\g$ is an extension of a $2p$-dimensional frobeniusian Lie algebra.
\end{corollary}

\noindent\textbf{Example: $\dim \g =9$.}
In case(C), $\g$is an extension of the following $8$-dimensional frobeniusian Lie algebra:
$$
\left\{
\begin{array}{l}
    d\om_2=  \om_2 \we \om_3 + \om_4 \we \om_5 + \om_{8} \we \om_{9} \\
    d\om_3=0 \\
    d\om_4=d_4 \om_3 \we \om_4 \\
    d\om_5=(-1-d_4) \om_3 \we \om_5 +(\alpha_1+\beta_2) \om_6 \we \om_9\\
    d\om_6=d_6 \om_3 \we \om_6 \\
    d\om_7=(-1-d_6) \om_3 \we \om_7 +\beta_2 \om_4 \we \om_9\\
    d\om_8=d_8 \om_3\we \om_8 + \alpha_1 \om_4 \we \om_6\\
    d\om_9=(-1-d_8) \om_3 \we \om_9 \\
    \mathrm{with} \ d_4+d_6-d_8=0 
\end{array}
\right.
$$
In case(D), $\g$ is an extension of the following $8$-dimensional frobeniusian Lie algebra: $$
\left\{
\begin{array}{l}
    d\om_2=  \om_2 \we \om_3 + \om_4 \we \om_5 + \om_6 \we \om_7 + \om_8\we \om_9 \\
    d\om_3=0 \\
    d\om_4=d_4 \om_3 \we \om_4+\alpha_3 \om_7 \we \om_8 \\
    d\om_5=(-1-d_4) \om_3 \we \om_5\\
    d\om_6=d_6 \om_3 \we \om_6+ \alpha_2 \om_5 \we \om_8 \\
    d\om_7=(-1-d_6) \om_3 \we \om_7 \\
    d\om_8=d_8 \om_3 \we \om_8 \\
    d\om_9=(-1-d_8) \om_3 \we \om_9+ (\alpha_3-\alpha_2) \om_5 \we \om_7 \\
    \mathrm{with} \  d_4+d_6-d_8+1=0.
\end{array}
\right.
$$

\noindent\textbf{A generalisation.} An element $(i,j,k) \in \mathbb{N}^3$ will be called a triplet of $ \mathcal{R}$, if the eigenvalues $\lambda_i,\lambda_j,\lambda_k$ satisfy a relation of $\mathcal{R}$. Two triplets $(i_1,i_2,i_3)$ and $(j_1,j_2,j_3)$ of $\mathcal{R}$ will be called to be unconnected if for any $k, j_k \notin (i_1,i_2,i_3)$. With the previous assumptions, if all triplets of $\mathcal{R}$ are unconnected, then the structural equations of $\g$ derive from the previous theorem.

\subsubsection{We assume that $r(\g)=p-3$ wuth two  connected triplets }
There are two types of connections between triplets that can be summarized by
$$\left\{
\begin{array}{l}
\lambda_i+\lambda_j=\lambda_k\\
\lambda_i+\lambda_k=\lambda_l
\end{array}
\right.
\ , \ \left\{
\begin{array}{l}
\lambda_i+\lambda_j=\lambda_k\\
\lambda_i+\lambda_s=\lambda_t\\
s,t \neq j,k.
\end{array}
\right.
$$

$\bullet$ In the first case, the simplest algebra is of dimension $11$, as above the general case will be deduced from this example. Our notations are always the same. The nonnull eigenvalues are $\pm \lambda_4,\pm\lambda_6,\pm\lambda_8,\pm\lambda_{10}$ with
$$\lambda_4+\lambda_6=\lambda_8, \ \lambda_4+\lambda_8=\lambda_{10}.$$
Then $\varphi_1$ is given by
$$\left\{
\begin{array}{l}
\varphi_1(e_2,e_3)=ae_2, \ \varphi_1(e_2,e_i)=c_ie_i, \ \varphi_1(e_3,e_1)=b_ie_i, \\
\varphi_1(e_{2i},e_{2i+1})=a_ie_2+b_ie_3, \ i\geq 2,\\
\varphi_1(e_4,e_6)=\alpha_1 e_8, \ \varphi_1(e_4,e_8)=\alpha_2 e_{10}, \ \varphi_1(e_4,e_9)= \beta_1e_7, \ \varphi_1(e_4,e_{11}=\beta_2e_0,\\
\varphi_1(e_5,e_7)=\beta_3 e_9, \ \varphi_1(e_5,e_8)=\alpha_3 e_{6}, \ \varphi_1(e_5,e_9)= \beta_4e_{11}, \ \varphi_1(e_5,e_{10}=\alpha_4e_8,\\
\varphi_1(e_6,e_9)=\beta_5e_5, \ \varphi_1(e_7,e_8)=\alpha_5e_4, \\
\varphi_1(e_8,e_{11})=\beta_4e_5, \ \varphi_1(e_9,e_{10}=\alpha_6e_4.
\end{array}
\right.
$$
The Jacobi conditions imply in particular
$$\left\{
\begin{array}{l}
a \neq 0, \ (a=1), \ a_i=1, \ b_i=0,\\
c_{2i}=-\lambda_{2i}, \ c_{2i+1}=\lambda_{2i},\\
d_{2i}+d_{2i+1}=-1
\end{array}
\right.
$$
We deduce
\begin{proposition}
With these hypothesis, $\g$ is an one dimensional extension of a frobeniusian Lie algebra.
\end{proposition}
The other conditions lead to the following cases:

\textit{First case:} $\alpha_i=\beta_i=0$ 

\textit{Second case:} $d_4+d_6-d_8=0,\alpha_i=\beta_i=0$

\noindent the other solutions of Jacobi’s equations lead to a change in Darboux’s basis in the second case and we obtain 
\begin{theorem}
Let $\g$ be  a $(11)$-dimensional diagonal contact Lie algebra  of rank  $2$. Then $\g$ is an extension of a $10$-dimensional frobeniusian Lie algebra and isomorphic to one of the following case
$$
\left\{
\begin{array}{l}
   d\om_1= \om_2 \we \om_3 + \om_4 \we \om_5 + \cdots + \om_{10} \we \om_{11} \\
    d\om_2=  \om_2 \we \om_3 + \om_4 \we \om_5 + \cdots + \om_{10} \we \om_{11} \\
    d\om_3=0 \\
    d\om_4=\left(\lambda_4\om_1-\lambda_4\om_2+d_4 \om_3\right) \we \om_4 \\
    d\om_5=\big(-\lambda_4\om_1+\lambda_4\om_2+(-1-d_4) \om_3\big) \we \om_5\\
    d\om_6=\left(\lambda_6\om_1-\lambda_6\om_2+d_6 \om_3\right) \we \om_6 \\
    d\om_7=\big(-\lambda_6\om_1+\lambda_6\om_2+(-1-d_6) \om_3\big) \we \om_7\\
    d\om_8=\left((\lambda_4+\lambda_6)(\om_1-\om_2)+d_8 \om_3\right) \we \om_8\\
    d\om_9=\big(-(\lambda_4-\lambda_6)(\om_1-\om_2)+(-1-d_8) \om_3\big) \we \om_9 \\
    d\om_{10}=\big((\lambda_4+\lambda_{8})(\om_1-\om_2)+d_{10} \om_3\big)\we \om_{10} \\
    d\om_{11}=\big(-(\lambda_4+\lambda_{8})(\om_1-\om_2)+(-1-d_{10}) \om_3\big)\we \om_{11} \\
\end{array}
\right.
$$
$$
\left\{
\begin{array}{l}
   d\om_1= \om_2 \we \om_3 + \om_4 \we \om_5 + \cdots + \om_{2p} \we \om_{2p+1} \\
    d\om_2=  \om_2 \we \om_3 + \om_4 \we \om_5 + \cdots + \om_{2p} \we \om_{2p+1} \\
    d\om_3=0 \\
    d\om_4=\left(\lambda_4\om_1-\lambda_4\om_2+d_4 \om_3\right) \we \om_4 \\
    d\om_5=\big(-\lambda_4\om_1+\lambda_4\om_2+(-1-d_4) \om_3\big) \we \om_5+ \beta_5 \om_5 \we \om_9  \\
    d\om_6=\left(\lambda_6\om_1-\lambda_6\om_2+d_6 \om_3\right) \we \om_6 \\
    d\om_7=\big(-\lambda_6\om_1+\lambda_6\om_2+(-1-d_6) \om_3\big) \we \om_7+ \beta_1 \om_4 \we \om_9  \\
    d\om_8=\left((\lambda_4+\lambda_6)(\om_1-\om_2)+d_8 \om_3\right) \we \om_8 +\alpha_1 \om_4 \we \om_6 \\
    d\om_9=\big(-(\lambda_4-\lambda_6)(\om_1-\om_2)+(-1-d_8) \om_3\big) \we \om_9 \\
    d\om_{10}=\big((\lambda_4+\lambda_{8})(\om_1-\om_2)+d_{10} \om_3\big)\we \om_{10} \\
    d\om_{11}=\big(-(\lambda_4+\lambda_{8})(\om_1-\om_2)+(-1-d_{10}) \om_3\big)\we \om_{11} \\
    d_4+d_6-d_8=0
\end{array}
\right.
$$
\end{theorem}
\noindent\textbf{Generalisation.} As before, this case is generalized to dimension $2p+1$, $p \geq 5$ provided that the triplets are not connected.

$\bullet$ Let’s study the second case, that is, when the triplets are of the type
$$\left\{
\begin{array}{l}
\lambda_i+\lambda_j=\lambda_k\\
\lambda_i+\lambda_s=\lambda_t\\
s,t \neq j,k.
\end{array}
\right.
$$
The smallest possible dimension of $\g$ is $13$. We can assume that relationships on eigenvalues are
$$\lambda_4+\lambda_6=\lambda_8, \ \ \lambda_4+\lambda_{10}=\lambda_{12}$$
(recall that $\lambda_i \neq \lambda_j \neq 0$). This implies that $\varphi_1(é_2,e_3) \neq 0$ and we can thake $\varphi_1(e_2,e_3)=e_2$. We deduce in particular
$$d\om_1=d\om_2=\sum_{i \leq 6} \om_52i)\we \om_{2i+1}, \ d\om_3=0$$
and $\g$ is a one dimensional extension of a frobeniusian Lie algebra. The other structure constants are given by
$$
\begin{array}{l}
d\om_4=(\lambda_4(\om_1-\om_2)+d_4\om_3)\we \om_4+ \alpha_{7,8}\om_7 \we \om_8 + \alpha_{11,12} \om_{11} \we \om_{12}\\
d\om_5=(-\lambda_4(\om_1-\om_2)+d_4\om_3)\we \om_5+ \alpha_{6,9}\om_6 \we \om_9 + \alpha_{10,13} \om_{10} \we \om_{13}\\
d\om_6=(\lambda_6(\om_1-\om_2)+d_6)\we \om_6+ \alpha_{5,8}\om_5 \we \om_8 \\ d\om_7=(-\lambda_6(\om_1-\om_2)+d_7)\we \om_7+ \alpha_{4,9}\om_4 \we \om_9\\
d\om_8=(\lambda_8(\om_1-\om_2)+d_8)\we \om_8+ \alpha_{4,6}\om_4 \we \om_6 \\ d\om_9=(-\lambda_6(\om_1-\om_2)+d_9)\we \om_9+ \alpha_{5,7}\om_5 \we \om_7\\
d\om_{10}=(\lambda_{10}(\om_1-\om_2)+d_{10})\we \om_{10}+ \alpha_{5,12}\om_5 \we \om_{12} \\ d\om_{11}=(-\lambda_{10}(\om_1-\om_2)+d_{11})\we \om_{11}+ \alpha_{4,13}\om_4 \we \om_{13}\\
d\om_{12}=(\lambda_{12}(\om_1-\om_2)+d_{12})\we \om_{12}+ \alpha_{4,10}\om_4 \we \om_{10} \\ d\om_{13}=(-\lambda_{12}(\om_1-\om_2)+d_{13})\we \om_{13}+ \alpha_{5,11}\om_5 \we \om_{11}\\
\end{array}
$$
i) If $d_4+d_6-d_8 \neq 0,1$ and $d_4+d_{10}-d_{12} \neq 0,1$ then for any $i,j$, $\alpha_{i,j}=0$ and $\g$ is a one dimensional extension of the model of frobeniusian Lie algebra (\cite{Go2}). 

ii) If $d_4+d_6-d_8 =0$ and $d_4+d_{10}-d_{12} \neq 0,1$ then all parameters $\alpha_{i,j}$ are $0$ except $\alpha_{4,6}$ and $\alpha_{4,9}$. Likewise si $d_4+d_6-d_8 \neq 0,1$ and $d_4+d_{10}-d_{12} =0$. in this case $\alpha_{4,13}=\alpha_{4,10}=0$, but this case is isomorphic to the previous one.  

iii) We also recover an isomorphic situation when $d_4+d_6-d_8 =1 $ and $d_4+d_{10}-d_{12} \neq 0,1$ or when $d_4+d_6-d_8 \neq 0,1$ and $d_4+d_{10}-d_{12} =1$. 

iv) If $d_4+d_6-d_8 =0$ and $d_4+d_{10}-d_{12} =0$ then the only possibly non-null parameters are $\alpha_{6,9},\alpha_{4,9},\alpha_{4,6}$ and $\alpha_{10,13},\alpha_{5,11},\alpha_{4,10}$. The other cases, namely $d_4+d_6-d_8 =0, d_4+d_{10}-d_{12} =1$, $d_4+d_6-d_8 =1, d_4+d_{10}-d_{12} =0$and $d_4+d_6-d_8 =1, d_4+d_{10}-d_{12} =1$ are isomorphic, via a Darboux change of basis, to the first.

\begin{theorem}
Let $\g$ be  a $(13)$-dimensional diagonal contact Lie algebra  of rank  $3$. Then $\g$ is an extension of a $10$-dimensional frobeniusian Lie algebra and isomorphic to one of the following case
$$
\left\{
\begin{array}{l}
   d\om_1= \om_2 \we \om_3 + \om_4 \we \om_5 + \cdots + \om_{12} \we \om_{13} \\
    d\om_2=  \om_2 \we \om_3 + \om_4 \we \om_5 + \cdots + \om_{12} \we \om_{13} \\
    d\om_3=0 \\
    d\om_4=(\lambda_4(\om_1-\om_2)+d_4 \om_3) \we \om_4 \\
    d\om_5=(-\lambda_4(\om_1-\om_2)+(-1-d_4) \om_3) \we \om_5\\
    d\om_6=(\lambda_6(\om_1-\om_2)+d_6 \om_3) \we \om_6 \\
    d\om_7=(-\lambda_6(\om_1-\om_2+(-1-d_6) \om_3) \we \om_7\\
    d\om_8=\left((\lambda_4+\lambda_6)(\om_1-\om_2)+d_8 \om_3\right) \we \om_8\\
    d\om_9=\big(-(\lambda_4-\lambda_6)(\om_1-\om_2)+(-1-d_8) \om_3\big) \we \om_9 \\
     d\om_{10}=(\lambda_{10}(\om_1-\om_2)+d_{10} \om_3)\we \om_{10} \\
    d\om_{11}=(-\lambda_{10})(\om_1-\om_2)+(-1-d_{10}) \om_3)\we \om_{11} \\
    d\om_{12}=\big((\lambda_4+\lambda_{10})(\om_1-\om_2)+d_{12} \om_3\big)\we \om_{12} \\
    d\om_{13}=\big(-(\lambda_4+\lambda_{10})(\om_1-\om_2)+(-1-d_{12}) \om_3\big)\we \om_{13} \\
\end{array}
\right.
$$
$$
\left\{
\begin{array}{l}
   d\om_1= \om_2 \we \om_3 + \om_4 \we \om_5 + \cdots + \om_{12} \we \om_{13} \\
    d\om_2=  \om_2 \we \om_3 + \om_4 \we \om_5 + \cdots + \om_{12} \we \om_{13} \\
    d\om_3=0 \\
    d\om_4=(\lambda_4(\om_1-\om_2)+d_4 \om_3) \we \om_4 \\
    d\om_5=(-\lambda_4(\om_1-\om_2)+(-1-d_4) \om_3) \we \om_5+\alpha_{6,9}\om_6 \we \om_9+ \alpha_{10,13} \om_{10} \we \om_{13}\\
    d\om_6=(\lambda_6(\om_1-\om_2)+d_6 \om_3) \we \om_6 \\
    d\om_7=(-\lambda_6(\om_1-\om_2+(-1-d_6) \om_3) \we \om_7+\alpha_{4,9}\om_4 \we \om_9\\
    d\om_8=\left((\lambda_4+\lambda_6)(\om_1-\om_2)+d_8 \om_3\right) \we \om_8+\alpha_{4,6}\om_4 \we \om_6\\
    d\om_9=\big(-(\lambda_4-\lambda_6)(\om_1-\om_2)+(-1-d_8) \om_3\big) \we \om_9 \\
     d\om_{10}=(\lambda_{10}(\om_1-\om_2)+d_{10} \om_3)\we \om_{10} \\
    d\om_{11}=(-\lambda_{10})(\om_1-\om_2)+(-1-d_{10}) \om_3)\we \om_{11}+ \alpha_{4,13} \om_{4} \we \om_{13} \\
    d\om_{12}=\big((\lambda_4+\lambda_{10})(\om_1-\om_2)+d_{12} \om_3\big)\we \om_{102}+ \alpha_{4,10} \om_{4} \we \om_{10} \\
    d\om_{13}=\big(-(\lambda_4+\lambda_{10})(\om_1-\om_2)+(-1-d_{12}) \om_3\big)\we \om_{13} \\
    d_4+d_6-d_8=0
\end{array}
\right.
$$
\end{theorem}

\section{How to approach the general case} It is not for to day.
\subsection{A schedule}
Complete the diagonal case study according to the rank of $\g$.

Focus on the reduced case. This means that $f$ is written in a Darboux basis as $f=f_s+f_n$ where $f_s$ is diagonal and $f_n$ is nilpotent and reduced in the form of Jordan. As $f_s$ and $f_n$ are derivations of $\varphi_1$, we determine this application using the diagonal case and write that $f_n$ is also a derivation.

Finally, there is the delicate aspect of reducing an endomorphism modulo the Lie group $R_p$. To illustrate this difficulty, we use the case of dimension $5$ from this perspective. 

\subsection{$5$-dimensional Contact Lie algebras, case $\K=\C$}
We saw that the matrix of $f$ could be written in the Darboux basis as
$$\begin{pmatrix}
       0&   b_{12}&-b_{42} & b_{32} \\
    0 &  0& 0 &0\\
     0&   b_{32} & b_{33}&   b_{34} \\
     0&   b_{42} & b_{43}&  - b_{33} 
\end{pmatrix} 
$$
We can make another discount by keeping a base of Darboux. If $b_{34} \neq 0$, a change of basis 
$Y_4=e_4-\frac{b_{33}}{b_{34}}, Y_5=e_5$ gives $b_{33}=0.$ We obtain the following matrices
$$F_1 \label{F1}
=\begin{pmatrix}
       0&   b_{12}&-b_{42} & 0\\
    0 &  0& 0 &0\\
     0&   0& 0&   b_{34} \\
     0&   b_{42} & b_{43}& 0\end{pmatrix} 
 \ \ \ \ {\rm or} \ \ 
F_2=\begin{pmatrix}
       0&   b_{12}&-b_{42} & 0\\
    0 &  0& 0 &0\\
     0&   0& b_{33}&   0\\
     0&   b_{42} & b_{43}&  - b_{33} 
\end{pmatrix} 
$$
\subsubsection{$f$ is diagonal}
If $f \neq 0$, we have $r(\g)=1$ and we have the following contact Lie algebras We deduce the following Lie algebras
\begin{enumerate}[label=\Alph*)]
  \item $ 
  \left\{\begin{array}{l}
     d\om_1=\om_2 \we \om_3 + \om_4 \we \om_5  \\
     d\om_2 =   \om_2 \we \om_3 \\
     d\om_3=0 \\
     d\om_4= (\lambda_4\om_1-\lambda_4\om_2 + d_4 \om_3) \we \om_4 \\ 
      d\om_5= (-\lambda_4\om_1+\lambda_4\om_2 - d_4 \om_3) \we \om_5 \\ 
\end{array}
\right. $
  \item  ($ \g $ is  an  extension  of  a  frobeniusian Lie algebra ).
  
  $ 
  \left\{\begin{array}{l}
     d\om_1=\om_2 \we \om_3 + \om_4 \we \om_5  \\
     d\om_2 =   \om_2 \we \om_3 + \om_4 \we \om_5 \\
     d\om_3=0 \\
     d\om_4= (\lambda_4\om_1-\lambda_4\om_2 +d_4\om_3) \we \om_4 \\ 
      d\om_5= (-\lambda_4\om_1+\lambda_4\om_2 - (1+d_4) \om_3) \we \om_5 \\ 
\end{array}
\right. $
\end{enumerate}
\subsubsection{$f$ is reduced}
If $rk(\g)=1$ then $f$ is written
$$\begin{pmatrix}
       0&  1&0&0 \\
    0 &  0& 0 &0\\
     0&   0 & \lambda_4& 0\\
     0&   0& 0&  -\lambda_4
\end{pmatrix} 
$$
If not, assuming that $f \neq 0$, we have the following cases for the matrix of $f$
$$\begin{pmatrix}
       0&  1&0&0 \\
    0 &  0& 1 &0\\
     0&   0 & 0& 1\\
     0&   0& 0&0
\end{pmatrix} , \ \begin{pmatrix}
       0&  0&0&0 \\
    0 &  0& 1 &0\\
     0&   0 & 0& 1\\
     0&   0& 0&0
\end{pmatrix} , \ \begin{pmatrix}
       0&  1&0&0 \\
    0 &  0& 0 &0\\
     0&   0 & 0& 1\\
     0&   0& 0&0
\end{pmatrix} , \ \begin{pmatrix}
       0&  0&0&0 \\
    0 &  0& 0 &0\\
     0&   0 & 0& 1\\
     0&   0& 0 &0
\end{pmatrix} , \ 
$$
$\bullet$ In the first case, we have
$$\left\{
\begin{array}{l}
      \varphi_1(e_2,e_3)=ae_2    \\
          \varphi_1(e_3,e_4)=be_4, \ \varphi_1(e_3,e_5)=ce_5, \\
     \varphi_1(e_4,e_5)=de_2.
\end{array}
\right.
$$
But the conditions $\delta_{\mu_0}\varphi_1=0, \varphi_1 \circ \varphi_1=-\delta_{\mu_0}\varphi_2$ implies $b+c+d=0$ and $\alpha=0$ and we have a contradiction.

$\bullet$ The second and third cases lead to a contradiction ($d(d\om_1) \neq 0$).

$\bullet$ \textit{Fourth case. }
This implies
$$\left\{
\begin{array}{l}
      \varphi_1(e_2,e_3)=ae_2+be_4,    \\
     \varphi_1(e_2,e_4)=0 , \ \varphi_1(e_2,e_5)=ee_2+fe_4,  \\
     \varphi_1(e_3,e_4)=ge_2+he_4 \ \varphi_1(e_3,e_5)=(e+g)e_3+(f+h)e_5+ me_2+pe_4, \\
     \varphi_1(e_4,e_5)=ke_2+le_4.
     \end{array}
\right.
$$
Then $I=\K\{e_1,e_2,e_4\}$ is an abelian ideal and $\g$ is a semi-direct product of $I$ by a $2$ dimensional solvable Lie algebra. 
Using a  change of Darboux basis, we can assume that $d\om_5=0$ and $d\om_3=(e+g)\om_3 \we \om_5$ that is $f+h=0$. This implies $k=f$ and $g=b-2e$. We can also reduce $\varphi_1(e_2,e_3)=ae_2+be_3.$ Si $b \neq 0$, we can take $a=0,b=1$. Othewise we have the cases $a=1,b=0$ and $a=b=0$.

 \medskip
 
\noindent i) $b=1,a=0$. The resolution of the equations $\varphi_1 \circ \varphi_1 +\delta_{\mu_0}\varphi_2=0, f \in Der(\varphi_1)$ gives the following Lie algebras
 \begin{enumerate}[label=\Alph*),resume]
\item $ 
  \left\{\begin{array}{l}
     d\om_1=\om_2 \we \om_3 + \om_4 \we \om_5  \\
     d\om_2 =   \om_3 \we \om_4 +\om_1\we \om_3\\
     d\om_3=\om_3 \we \om_5 \\ 
     d\om_4= \om_2 \we \om_3 +\om_1 \we \om_5\\ 
      d\om_5= 0 \\ 
\end{array}
\right. $
\item $ 
  \left\{\begin{array}{l}
     d\om_1=\om_2 \we \om_3 + \om_4 \we \om_5  \\
     d\om_2 = \om_2 \we \om_5 - \om_3 \we \om_4 +\om_1\we \om_3\\
     d\om_3=0 \\ 
     d\om_4= \om_2\we  \om_3 +\om_1 \we \om_5\\ 
      d\om_5= 0 \\ 
\end{array}
\right. $\item $ 
  \left\{\begin{array}{l}
     d\om_1=\om_2 \we \om_3 + \om_4 \we \om_5  \\
     d\om_2 =  \om_2 \we \om_5 - \om_3 \we \om_4 +f \om_4 \we \om_5 + \om_1\we \om_3\\
     d\om_3=0 \\ 
     d\om_4= \om_2 \we \om_3 +f (\om_2 \we \om_5 - \om_3 \we \om_4 +f \om_4 \we \om_5 )+\om_1 \we \om_5\\ 
      d\om_5= 0 \\ 
\end{array}
\right. $
\end{enumerate}

 \medskip
 
\noindent ii) $a=1,b=0$. The resolution of the equations $\varphi_1 \circ \varphi_1 +\delta_{\mu_0}\varphi_2=0, f \in Der(\varphi_1)$ implies $e=0$ and $f^2-f-1=0$ and this gives the following Lie algebras
 \begin{enumerate}[label=\Alph*),resume]
\item $ 
  \left\{\begin{array}{l}
     d\om_1=\om_2 \we \om_3 + \om_4 \we \om_5  \\
     d\om_2 =  \om_2 \we \om_3+f \om_4 \we \om_5 +\om_1\we \om_3\\
     d\om_3=0 \\ 
     d\om_4= f\om_2 \we \om_5-f \om_3 \we \om_4+l \om_4 \we \om_5 +\om_1 \we \om_5\\ 
      d\om_5= 0 \\ 
\end{array}
\right. $
with $ f^2-f-1=0.$
\end{enumerate}

 \medskip
 
\noindent iii) : $a=0,b=0$. The resolution of the equations $\varphi_1 \circ \varphi_1 +\delta_{\mu_0}\varphi_2=0, f \in Der(\varphi_1)$ implies $e=0$ and $f^2-1=0$ and we obtain
\begin{enumerate}[label=\Alph*),resume]
\item $ 
  \left\{\begin{array}{l}
     d\om_1=\om_2 \we \om_3 + \om_4 \we \om_5  \\
     d\om_2 = f \om_4 \we \om_5 +\om_1\we \om_3\\
     d\om_3=0 \\ 
     d\om_4= f\om_2 \we \om_5-f \om_3 \we \om_4+l \om_4 \we \om_5 +\om_1 \we \om_5\\ 
      d\om_5= 0 \\ 
\end{array}
\right. $
with $ f^2-1=0.$
\end{enumerate}

$\bullet$ \textit{Last case corresponding to $f(e_5)=e_4.$} Since $f$ is a derivation of $\varphi_1$, we have
$$f(\varphi_1(e_2,e_5)=\varphi_1(e_2,e_4), f(\varphi_1(e_3,e_5)=\varphi_1(e_3,e_4)$$
and $\varphi_1(e_2,e_3), \varphi_1(e_2,e_4), \varphi_1(e_4,e_5)\in {\rm ker} f$ and
$$
\left\{
\begin{array}{l}
      \varphi_1(e_2,e_3)=ae_2+be_3+ce_4    \\
       \varphi_1(e_2,e_4)=qe_4    \\  
        \varphi_1(e_2,e_5)=ee_2+fe_3+ge_4+qe_5    \\
         \varphi_1(e_3,e_4)=he_4    \\  
          \varphi_1(e_3,e_5)=je_2+ke_3+le_4+he_5    \\
           \varphi_1(e_4,e_5)=me_2+ne_3+pe_4   \\
           \end{array}
           \right.
           $$
Rather than solving the equations concerning $\varphi_1$, we can solve the structural equations $d(d\om_i)=0.$ For example, we have 
$d\om_5=h\om_3 \we \om_5 + qd \om_2 \we \om_5$. If  $q\neq 0$ , we can reduce this equation to $d\om_5= \om_2 \we \om_5$ ($q=1,h=0$) and $d(d\om_2)= 0$ is equivalent to $a=0$. In this case we find
\begin{enumerate}[label=\Alph*),resume]
\item $ 
  \left\{\begin{array}{l}
    d\om_1=\om_2 \we \om_3 + \om_4 \we \om_5  \\
     d\om_2 = c \om_2 \we \om_5\\
     d\om_3=2 \om_2 \we \om_3 + f\om_2 \we \om_5 +2 \om_4 \we \om_5\\ 
     d\om_4=  c\om_2 \we \om_3+ \om_2 \we \om_4 +g\om_2 \we \om_5-\frac{1}{2}\om_3 \we \om_5+c\om_4 \we \om_5+\om_1\we \om_5 \\
      d\om_5=  \om_2 \we \om_5.\\ 
\end{array}
\right. $
\end{enumerate}
If $q=0$ and $h \neq 0$, then $b=n=0$ and $f(a-1)=0$. If $a=1$ then $d(d\om_2) \neq 0$. Then $f=0$ we obtain the following Lie algebra
\begin{enumerate}[label=\Alph*),resume]
\item $ 
  \left\{\begin{array}{l}
    d\om_1=\om_2 \we \om_3 + \om_4 \we \om_5  \\
     d\om_2 = -2 \om_2 \we \om_3 + j\om_3 \we \om_5 -2 \om_4 \we \om_5\\
     d\om_3=c \om_3 \we \om_5\\ 
     d\om_4=  c\om_2 \we \om_3+(c^2 +\frac{1}{2})\om_2 \we \om_5 +\om_3 \we \om_4+l\om_3 \we \om_5+3c\om_4 \we \om_5+\om_1\we \om_5 \\
      d\om_5=  \om_2 \we \om_5.\\ 
\end{array}
\right. $
\end{enumerate}
If $q=h=0$, then if $a$ or $b$ are not zero, then $c=e+k$ and $jb-ak=be-af=0$ and we can reduce the equation putting $\om_3'=a\om_3-b\om_2$ (we assume $a \neq 0$) to $d\om_3=0$ and we obtain the following Lie algebra 
\begin{enumerate}[label=\Alph*),resume]
\item $ 
  \left\{\begin{array}{l}
    d\om_1=\om_2 \we \om_3 + \om_4 \we \om_5  \\
     d\om_2 = \om_2 \we \om_3  +j \om_3 \we \om_5\\
     d\om_3=0\\ 
     d\om_4=  - \om_2 \we \om_5 +l\om_3 \we \om_5+p\om_4 \we \om_5+\om_1\we \om_5 \\
      d\om_5= 0.\\ 
\end{array}
\right. $
\end{enumerate}
Finally if $q=h=a=b=0$ then $c=e+k\neq 0$ and  we obtain
\begin{enumerate}[label=\Alph*),resume]
\item $ 
  \left\{\begin{array}{l}
    d\om_1=\om_2 \we \om_3 + \om_4 \we \om_5  \\
     d\om_2 = e\om_2 \we \om_3  +j \om_3 \we \om_5\\
     d\om_3=f\om_2 \we \om_3  +(1-e)\om_3 \we \om_5\\
     d\om_4=  g \om_2 \we \om_5 +l\om_3 \we \om_5+\om_1\we \om_5 \\
      d\om_5= 0.\\ 
\end{array}
\right. $
\end{enumerate}

\subsubsection{$f$ is diagonalizable, non reduced} Recall that in the Darboux basis, after choosing the vector $e_2$ in the $\ker f$, is
$$
\begin{pmatrix}
0 & b_{12} & -b_{42} & b_{32} \\
0 & 0 & 0 & 0 \\
0 & b_{32} & b_{33} & b_{34}\\
0 & b_{42} & b_{43} & -b_{33}
\end{pmatrix}
$$
Recall also, since we assume $f \neq 0$ and diagonalizable, the eigenvalues  $f$ are $0,0,\lambda,-\lambda$ with $\lambda \neq 0.$
\begin{enumerate}
\item If $b_{32} \neq 0$ we can consider $b_{32}=0$ and $b_{42}=0$ ($e'_4=e_4+b_{42}e_5,e'_5=e_5)$. Since the multiplicity of $0$ is $2$, then $\delta=-b_{33}^2-b_{34}b_{43} \neq 0$ and $\Delta=b_{43}+b_{12}\delta=0.$ If $b_{43}=0$, then $b_{12}=0$ and $f$ can be reduced in a diagonal form modulo the Lie group $R_2$ (which transforms a Darboux basis in a Darboux basis). If $b_{43}\neq 0$, then $b_{12} \neq 0$. In this case $\ker f$ is generated by $e_2$ and $X_3=e_3+\frac{b_{33}}{\delta}e_4-b_{12}e_5$. 
The new basis $\{e_2,X_3,X_4=e_4-b_{12}e_2,X_5=e_5+\frac{b_{33}}{\delta}e_2\}$ is a Darboux basis and $f$ is also diagonalizable modulo $R_2$. 
\item If $b_{32}=0$ and $b_{34} \neq 0$, with a change of Darboux basis ($X_4=e_4-\frac{b_{33}}{b_{34}}e_5,X_5=e_5$), we can consider that $b_{33}=0$. Since $0$ is an eigenvalue of order $2$, $\delta=b_{12}b_{43}+b_{42}^2 =0.$ The case $b_{42}=0$ implies that $f$ is diagonalizable in a Darboux basis (if $b_{12}=0$), or the order of $0$ is $4.$ So we assume that $b_{42} \neq 0$. In this case $b_{12}b_{43} \neq 0$ We can assume that $b_{12}=1$. In this case the change of basis $\{X_2=e°2,X_3=e_3+e_4,X_4=e_4;X_5=e_5+e_2\}$ is in $R_2$ and the matrix of $f$ is
$$
F_1=
\begin{pmatrix}
0 & 0 & 0 & 0\\
0 & 0 & 0&0\\
0&0&0&b_{34}\\
0&0&-1&0
\end{pmatrix} 
$$
and this matrix can be reduced in a Darboux basis.
\item If $b_{32}=0$ and $b_{34} = 0$, since $0$ s of order $2$, then $b_{33}b_{12}=0.$ 
\begin{itemize}
\item If $b_{33} \neq 0, b_{12}=0$ and $v_2=e_2,v_3=e_3-\frac{b_{42}}{b_{33}}e_5, v_4=-\frac{b_{42}}{b_{33}}e_2+e_4+\frac{b_{43}}{2b_{33}}e_5$ is a basis of eigenvectors and a Darboux basis. 
\item If $b_{33}=0$, then $f_s=0.$ 
\end{itemize}
\end{enumerate}
Conclusion: If $f$ is diagonalizable, then $f$ can be reduced in a diagonal form modulo $R_2$.

\subsubsection{$f$ is not diagonalizable and not reduced}
Recall that if $f=f_s+f_n$ is the Jordan decomposition of $f$, then $f_s$ and $f_n$ also are derivations of $\varphi_1$. We have seen that, in dimension $5$, $f_s$ can be always reduced modulo $R_2$. We can therefore assume that $f_s$ is already in diagonal form. Thus, either $f_s=0$, or there exists $\lambda \neq 0$ such that $f(e_2)=f(e_3)=0,f(e_4=\lambda e_4, f(5)=-\lambda e_5$.

i) If $f_S \neq 0$, since $f_n$ commutes with $f_s$, the matrix of $f_n$ is
$$
\begin{pmatrix}
0 & b_{12} & 0 & 0\\
0 &0 & 0 & 0\\
0& 0& b_{33} & 0\\
0 & 0 & 0 & -b_{33}
\end{pmatrix}
$$
and the nilpotency of $f_n$ implies $b_{33}=0$ and $f_n$ is reduced.

ii) Assume that $f_s=0.$ Then $f_n$ nilpotent is equivalent to $b_{33}^2+b_{43}b_{34}=0.$

$\bullet$ If $b_{33}\neq 0$, there exists $\rho \neq 0$ such as $b_{34}=\rho b_{33}, b_{43}=\rho^{-1}b_{33}.$ The change of Jordan basis $(X_5=e_5-\rho e_4, X_i=e_i)$ allows us to consider that $b_{33}=0.$ 

$\bullet$ $b_{33} = 0.$ If $b_{34}=b_{43}=0$ and if $f_n$ is not reduced, its matrix is
$$
\begin{pmatrix}
0 & 0 & 0 & b_{32}\\
0 &0 & 0 & 0\\
0& b_{32}&0 & 0\\
0 & 0 & 0 & 0
\end{pmatrix}
$$
We deduce 
$$
\left\{
\begin{array}{l}
\varphi_1(e_2,e_3)=ae_2+be_4, \ \varphi_1(e_2,e_4)=0, \ \varphi_1(e_2,e_5)=ce_2+de_4,\\
\varphi_1(e_3,e_4)ee_2+fe_4, \ \varphi_1(e_4,e_5)=ge_2+he_4,\\
\varphi_1(e_3,e_5)=(g-a)e_5+(h-b)e_3+ie_2+je_4
\end{array}
\right.
$$
This shows that $I=\K\{e_1,e_2,e_4\}$ is an abelian ideal of $\g$. 

i) Assume that $[e_3,e_5]=0$. In this case, we can assume that $b=0$ (if $a \neq 0$ we consider $\tilde{\omega_4}-\om_4-\frac{b}{a}\om_2, \tilde{\om_3}=\om_3+ \frac{b}{a}\om_5$). We obtain
 $$ \mathrm{L})
  \ \ 
  \left\{\begin{array}{l}
    d\om_1=\om_2 \we \om_3 + \om_4 \we \om_5  \\
     d\om_2 = a\om_2 \we \om_3  +e \om_3 \we \om_43 + a \om_4 \we \om_5 + \om_1 \we \om_5\\
     d\om_3=0\\
     d\om_4=  -a\om_3 \we \om_4 +d\om_2 \we \om_5+\om_1\we \om_3 \\
      d\om_5= 0.\\ 
      ed+1=0
\end{array}
\right. $$

ii) Assume that $[e_3,e_5]=e_3$. This case leads to a contradiction.

\noindent Remark. A complete classification of $5$-dimensional complex Lie algebra was made by Diatta and Foreman \cite{Di}, using a completely different approach. They consider first the unimodular and indecomposable case. Next they study the case solvable and nonsolvable.


\begin{thebibliography}{99}



\bibitem{Salg1} Alvarez M.A., RodrÃ­guez-Vallarte M. C.  Salgado G. Low Dimensional Contact Lie Algebras. {\it Journal of Lie Theory}
 {\bf 29} (2019) 811-838.

\bibitem{Ba} Bagarello F., Bavuma Y., Russo F.G. Consequences of the Grunewald-O'Halloran conjecture for
algebras of pseudoquonic operators. {\it Forum Mathematicum}. To appear (2025).

 \bibitem{Bah} Bahturin Y., Goze M. $ \Z_2\times \Z_2$ -symmetric spaces
{\it Pacific J. Math.} {\bf 236} (2008), no. 1, 1-21.

\bibitem{Bay} Liu J., Sheng Y., Bai C. Maurer-Cartan characterizations and cohomologies of compatibles Lie algebras.  Sci. China, Math. 66, No. 6, 1177-1198 (2023)..

\bibitem{Co} Coll V.E., Russoniello N. Classification of contact seaweeds. {\it J. Algebra} {\bf 659}, 811-817 (2024).

\bibitem{Di}Diatta A.,Foreman B. Lattices in contact Lie groups and 5-dimensional conact sovmanifolds. Koday Math.J. 38 (2015), 228-248.


\bibitem{Go} Goze M. Sur la classe des  formes et syst\`emes invariants \`a gauche sur un groupe de Lie. (French) 
{C. R. Acad. Sci., Paris}, Ser. A 283, 499-502 (1976).

\bibitem{Go2} Goze M. 
Mod\`eles d'Alg\`ebres de Lie frobeniusiennes. 
C. R. Acad. Sci., Paris, Sér. I 293, 425-427 (1981).

\bibitem{G.H} Goze M.,  Khakimdjanov, Y. Nilpotent and Solvable Lie Algebras. In: M. Hazewinkel
(Ed.), Handbook of Algebra, Vol. 2, 615-663, North-Holland, Amsterdam. (2000).

\bibitem{G.H2}  Goze M.,  Khakimdjanov, Y. Nilpotent  Lie Algebras. Mathematics and its Applications (Dordrecht). 361. Dordrecht: Kluwer Academic Publishers. xv, 334 p. (1996).

\bibitem{G.R.contact} Goze M., Remm E. Contact and frobeniusian Lie algebras. {\it Differential Geom. Appl.} {\bf 35} (2014), 74-94.

\bibitem{G.R.gamma} Goze M., Remm E.  Riemannian  $\Gamma$ -symmetric spaces
World Scientific Publishing Co. Pte. Ltd., Hackensack, NJ, 2009, 195-206.
ISBN: 978-981-4261-16-6; 981-4261-16-5.

\bibitem{La} Ladra  M., Leite da Cunha B., Lopes S.A.  A classification of nilpotent compatible Lie algebras. Arxiv 2406.64036 (2024).









 \bibitem{Mak} Makarenko Y.N. Lie type algebras with an automorphism of finite order. J. Algebra 439, 33-66 (2015).

 
 \bibitem{Salg2} Salgado-Gonz\`alez G.  Invariants of contact Lie algebras. {\it Journal of Geometry and Physics} {\bf 144} (2019) 388-396. 


\bibitem{RV} Rodriguez-Vallarte M.M., Salgado G., Sanchez-Valenzuela O.A. On extensions of Frobenius-K\"ahler and Sasakian Lie algebras. Preprint (2025).

\end{thebibliography}
\end{document}